\documentclass[12pt]{amsart}
\usepackage{amssymb,mathrsfs}
\usepackage[all]{xy}

\usepackage[shortlabels]{enumitem}

\usepackage[sort&compress]{natbib} 

\usepackage{url}
\usepackage{verbatim}
\usepackage{comment}
\usepackage{geometry}
\usepackage{mathtools}
\usepackage{comment}
\usepackage{caption}
\DeclareCaptionLabelFormat{mywraplabel}{\hspace{.5ex}#1~#2.\newline}

\usepackage{rotating}
\usepackage{tikz,graphicx}
\tikzstyle{vertex}=[circle, draw, inner sep=0pt, minimum size=4pt]
\newcommand{\vertex}{\node[vertex]}
\usetikzlibrary{matrix,shapes,arrows,fit,calc}
\usetikzlibrary{graphs,graphs.standard}
\usetikzlibrary{decorations.pathmorphing}
\usetikzlibrary{shapes.callouts}

\definecolor{Cerulean}{cmyk}{0.94,0.11,0,0}
\definecolor{ForestGreen}{cmyk}{0.91,0,0.88,0.12}
\definecolor{RedViolet}{cmyk}{0.07,0.90,0,0.34}

\newtheorem{theorem}{Theorem}[section]
\newtheorem{lemma}[theorem]{Lemma}

\newtheorem{proposition}[theorem]{Proposition}
\newtheorem{corollary}[theorem]{Corollary}

\theoremstyle{definition}
\newtheorem{definition}[theorem]{Definition}

\newtheorem{example}[theorem]{Example}

\newtheorem{remark}[theorem]{Remark}

\newcommand\ba{\mathbf{a}}
\newcommand\bb{\mathbf{b}}

\newcommand\be{\mathbf{e}}

\newcommand\bp{\mathbf{p}}

\newcommand\bv{\mathbf{v}}

\newcommand\Z{\mathbb{Z}}
\newcommand\R{\mathbb{R}}

\newcommand\calF{\mathcal{F}}

\newcommand\calM{\mathcal{M}}

\newcommand\calP{\mathcal{P}}

\newcommand\calT{\mathcal{T}}

\newcommand\calV{\mathcal{V}}

\newcommand{\scrL}{\mathscr{L}}

\newcommand\inedge{\mathrm{in}}
\newcommand\outedge{\mathrm{out}}
\newcommand\In{\mathrm{In}}
\newcommand\Out{\mathrm{Out}}

\newcommand{\DKK}{\mathrm{DKK}}
\newcommand{\flow}{f}
\newcommand{\cliques}{\mathcal{C}}
\newcommand{\cliquevs}{\mathcal{V}}

\newcommand{\cov}{\mathrm{cov}}

\newcommand{\floor}[1]{\left\lfloor#1\right\rfloor}

\DeclareMathOperator{\ccw}{ccw}
\DeclareMathOperator{\cw}{cw}

\DeclareMathOperator{\rot}{rot}

\DeclareMathOperator{\vol}{Vol}

\title{Flow polytopes for extensions of bipartite graphs}

\author[Braun]{Benjamin Braun}
\address{Department of Mathematics\\
         University of Kentucky\\
\url{https://sites.google.com/view/braunmath/}}
\email{benjamin.braun@uky.edu}

\author[Bruegge]{Kaitlin Bruegge}
\address{Department of Mathematics\\
         University of Cincinnati\\
\url{https://researchdirectory.uc.edu/p/brueggkn}}
\email{brueggkn@ucmail.uc.edu}

\author[Davis]{Robert Davis}
\address{Department of Mathematics\\
        Colgate University
\url{https://sites.google.com/site/rtda223}}
\email{rdavis@colgate.edu}

\author[Hanely]{Derek Hanely}
\address{Department of Mathematics\\
         Penn State Behrend\\
\url{https://behrend.psu.edu/person/derek-hanely}}
\email{derek.hanely@psu.edu}

\date{16 October 2025}

\thanks{The authors thank Alejandro Morales and Martha Yip for helpful conversations.}

\begin{document}

\begin{abstract}
The space of unit flows on a finite acyclic directed graph is a lattice polytope called the flow polytope of the graph.
Given a bipartite graph $G$ with minimum degree at least two, we construct two associated acyclic directed graphs: the extension of $G$ and the almost-degree-whiskered graph of $G$.
We prove that the normalized volume of the flow polytope for the extension of $G$ is equal to the number of matchings in the almost-degree-whiskered graph of $G$.
Further, we refine this result by proving that the Ehrhart $h^*$-polynomial of the flow polytope for the extension of $G$ is equal to the unsigned matching polynomial of the almost-degree-whiskered graph of $G$.
\end{abstract}

\maketitle

\section{Introduction and main results}

\subsection{Motivation}

Flow polytopes have been the subject of intense study for decades due to their importance in optimization, geometric combinatorics, algebraic geometry, and other areas.
A significant amount of research in recent years has been devoted to combinatorial properties of flow polytopes~\cite{baldoni-vergne,vonbell2024triangulations,framedtriangulations,caracolvolume,combflowmodel,braun2024equatorialflowtriangulationsgorenstein,braunmcelroyfullvolumes,DKK,generalpitmanstanley,cyclicorderflow,integerpointsflow,lidskii,gtflow,flowdiagonalharmonics,meszaros-morales,meszaros-morales-striker,morrisidentityflow,rietsch2024rootpolytopesflowpolytopes}.
A perennial problem regarding flow polytopes is to compute and understand their volumes, where in this work we consider the \emph{normalized volume} of a flow polytope, which is the (relative) Euclidean volume times the factorial of the dimension of the polytope.
It is well-known that the normalized volume of a lattice polytope $P$ is a positive integer, and a well-known refinement of the normalized volume of a lattice polytope is the Ehrhart $h^*$-polynomial of $P$.

The Lidskii formula~\cite{baldoni-vergne,lidskii,meszaros-morales} is arguably the most important formula expressing the volume of a flow polytope with arbitrary netflow vector, but the Lidskii formula is a sum over integer compositions of products of multinomial coefficients and evaluations of Kostant partition functions for various graphs.
Thus, while this is a beautiful and useful formula, it remains mysterious in practice.
In the case of the flow polytopes considered in this paper, the Lidskii formula collapses to a single summand, and gives the volume as the number of integer points in an associated flow polytope defined by a different netflow vector.
In some cases this makes the determination of the volume easier, but not always.
While there are some lovely combinatorial models known for computing flow polytope volumes~\cite{combflowmodel}, there remains a need to better understand the behavior of these volumes, especially as the volume relates to the combinatorics of the underlying DAG.

Motivated by this goal, given a bipartite graph $H$ with minimum degree at least two, we construct two associated acyclic directed graphs: the extension $G(H)$ and the almost-degree-whiskered graph $W(H)$.
In this work, we prove that the normalized volume of the flow polytope $\calF_1(G(H))$ is equal to the number of matchings in $W(H)$.
Further, we refine this result by proving that the Ehrhart $h^*$-polynomial of $\calF_1(G(H))$ is equal to the unsigned matching polynomial of $W(H)$.

\subsection{Main results}

Let \(H=(V,E)\) be a finite directed acyclic graph, or DAG for short, with vertex set $V$ and directed edge set $E$.
We allow \(H\) to have multiple edges.
For each $v\in V(H)$, let $\mathrm{in}(v)$ denote the incoming edges of $v$ and let $\mathrm{out}(v)$ denote the outgoing edges of $v$. 
The vertex $v$ is called a \emph{source} if $\mathrm{in}(v) =\emptyset$, and it is called a \emph{sink} if $\mathrm{out}(v) = \emptyset$. 
Any other vertices are called \emph{inner vertices}. 
Let \(S(H)\) denote the set of sources in \(H\) and let \(T(H)\) denote the set of sinks in \(H\).

\begin{definition}\label{def:extension}
Given a DAG \(H\), we denote by \(G=G(H)\) the DAG obtained by adding new vertices \(s\) and \(t\) and a pair of directed edges from \(s\) to \(i\) for each \(i \in S(H)\) as well a pair of directed edges from \(j\) to \(t\) for each \(j \in T(H)\).
We call \(G(H)\) the \emph{extension} of \(H\).
\end{definition}

When $H$ is a bipartite graph with shores $S(H)$ and $T(H)$, we consider $H$ to be a DAG with all edges directed from $S(H)$ to $T(H)$.
See Figure~\ref{fig:G(H)labels} for an example of an extension graph in the bipartite case.
We will use the following notational conventions for extensions of bipartite graphs.

\begin{definition}\label{def:bipextension}
For a bipartite graph $H$ having no vertices of degree zero or one and having bipartition $V(H)=S(H) \uplus T(H)$, in $G(H)$ we label the edges from the source \(s\in G(H)\) to \(i\in S(H)\) as \(\alpha_{1,i}\) and \(\alpha_{2,i}\), any edge in \(H\) from \(i\) to \(j\) as \(\beta_{i,j}\), and the edges from \(j\in T(H)\) to the sink \(t\in G(H)\) as \(\gamma_{j,1}\) and \(\gamma_{j,2}\).
\end{definition}

\begin{example}\label{ex:G(H)}
Let $H=K_{3,2}$ be the complete bipartite DAG on vertex set $\{1,2,3,4,5\}$ shown in Figure~\ref{fig:G(H)labels}.
The extension $G(K_{3,2})$ is also shown in Figure~\ref{fig:G(H)labels}.
Note that by our naming convention, in $G(K_{3,2})$ the two edges from $s$ to $2$ are referred to as $\alpha_{1,2}$ and $\alpha_{2,2}$, the edge from $1$ to $4$ is referred to as $\beta_{1,4}$, and the edges from $5$ to $t$ are $\gamma_{5,1}$ and $\gamma_{5,2}$.
\end{example}

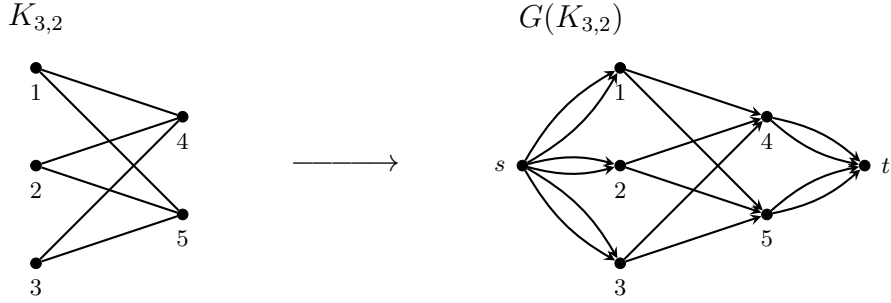
\begin{figure}[t]
\begin{center}
\begin{tikzpicture}[scale=1.3]
\begin{scope}
\node at (0,2.5) {$K_{3,2}$};
\vertex[fill,label=below:\scriptsize{$1$}](a1) at (0,2) {};\vertex[fill,label=below:\scriptsize{$2$}](a2) at (0,1) {};
\vertex[fill,label=below:\scriptsize{$3$}](a3) at (0,0) {};
\vertex[fill,label=below:\scriptsize{$4$}](a4) at (1.5,1.5) {};
\vertex[fill,label=below:\scriptsize{$5$}](a5) at (1.5,0.5) {};

\draw[thick] (a1)--(a4);
\draw[thick] (a1)--(a5);
\draw[thick] (a2)--(a4);
\draw[thick] (a2)--(a5);
\draw[thick] (a3)--(a4);
\draw[thick] (a3)--(a5);
\end{scope}
\begin{scope}[xshift=90]
\node at (0,1) {$\relbar\joinrel\relbar\joinrel\relbar\joinrel\relbar\joinrel\rightarrow$};
\end{scope}
\begin{scope}[xshift=170]
\node at (-.5,2.5) {$G(K_{3,2})$};
\vertex[fill,label=below:\scriptsize{$1$}](a1) at (0,2) {};\vertex[fill,label=below:\scriptsize{$2$}](a2) at (0,1) {};
\vertex[fill,label=below:\scriptsize{$3$}](a3) at (0,0) {};
\vertex[fill,label=below:\scriptsize{$4$}](a4) at (1.5,1.5) {};
\vertex[fill,label=below:\scriptsize{$5$}](a5) at (1.5,0.5) {};
\vertex[fill,label=left:\scriptsize{$s$}](s) at (-1,1) {};
\vertex[fill,label=right:\scriptsize{$t$}](t) at (2.5,1) {};

\draw[-stealth, thick] (a1)--(a4);
\draw[-stealth, thick] (a1)--(a5);
\draw[-stealth, thick] (a2)--(a4);
\draw[-stealth, thick] (a2)--(a5);
\draw[-stealth, thick] (a3)--(a4);
\draw[-stealth, thick] (a3)--(a5);

\draw[-stealth, thick] (s) to[bend left=15] (a1);
\draw[-stealth, thick] (s) to[bend right=15] (a1);
\draw[-stealth, thick] (s) to[bend left=15] (a2);
\draw[-stealth, thick] (s) to[bend right=15] (a2);
\draw[-stealth, thick] (s) to[bend left=15] (a3);
\draw[-stealth, thick] (s) to[bend right=15] (a3);
\draw[-stealth, thick] (a4) to[bend left=15] (t);
\draw[-stealth, thick] (a4) to[bend right=15] (t);
\draw[-stealth, thick] (a5) to[bend left=15] (t);
\draw[-stealth, thick] (a5) to[bend right=15] (t);
\end{scope}
\end{tikzpicture}
\end{center}
\caption{The graph extension for $K_{3,2}$.}
\label{fig:G(H)labels}
\end{figure}

We next define the operation of whiskering a graph, which is a special case of a generalization of the corona operation introduced by Frucht and Harary~\cite{fruchthararycorona}.
Recall that the \emph{join} of a vertex $v$ with a graph $H$, denoted $v\vee H$, is obtained by adding an edge between every vertex in $H$ and $v$.

\begin{definition}
Let $G$ be a finite simple graph and $\mathcal{H}:=\{H_v:v\in G\}$ a family of finite simple graphs.
The \emph{$\mathcal{H}$-corona of $G$}, denoted $G\odot \mathcal{H}$, is
\[
    G\odot \mathcal{H} := G \cup \left(\bigcup_{v \in V(G)} v \vee H_v\right).
\]
\end{definition}

\begin{example}\label{ex: H-corona}
    Let $G$ be the path on $\{1,2,3\}$ where $\deg(2) = 2$, and suppose $\mathcal{H} = \{H_1,H_2,H_3\}$ where $H_i = K_i$. 
    Then $G \odot \mathcal{H}$ is the graph illustrated in Figure~\ref{fig: H-corona}.
\end{example}

\begin{figure}[ht]
    \begin{center}
\begin{tikzpicture}[scale=2]
\begin{scope}
\vertex[fill,label=above:\scriptsize{$1$}](a1) at (0,0) {};\vertex[fill,label=above:\scriptsize{$2$}](a2) at (1,0) {};
\vertex[fill,label=above:\scriptsize{$3$}](a3) at (2,0) {};
\vertex[fill](a11) at (0,-1) {};
\vertex[fill](a21) at (0.7,-1) {};
\vertex[fill](a22) at (1.3,-1) {};
\vertex[fill](a31) at (1.7,-1) {};
\vertex[fill](a32) at (2.3,-1) {};
\vertex[fill](a33) at (2,-0.66) {};

\draw[thick] (a1)--(a2);
\draw[thick] (a1)--(a11);
\draw[thick] (a2)--(a3);
\draw[thick] (a2)--(a21);
\draw[thick] (a2)--(a22);
\draw[thick] (a21)--(a22);
\draw[thick] (a3)--(a31);
\draw[thick] (a3)--(a32);
\draw[thick] (a3)--(a33);
\draw[thick] (a31)--(a32);
\draw[thick] (a31)--(a33);
\draw[thick] (a32)--(a33);
\end{scope}
\end{tikzpicture}
\end{center}
    \caption{The $\mathcal{H}$-corona of $G$ described in Example~\ref{ex: H-corona}.}
    \label{fig: H-corona}
\end{figure}
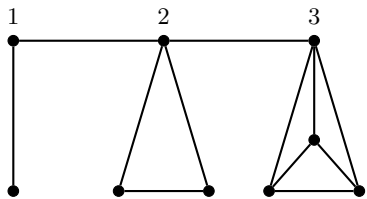

A special case of the $\mathcal{H}$-corona construction is the almost-degree-whiskering construction.

\begin{definition}
For a graph $G$, the \emph{almost-degree-whiskered graph} of $G$, denoted $W(G)$, is given by $W(G) \cong G \odot \mathcal{E}$, where
\[
    \mathcal{E} = \{ H_v:=\overline{K}_{\deg(v)-1} \mid v \in V(G)\}
\]
and $\overline{K}_k$ is the empty graph on $k$ vertices.
\end{definition}

Our first main result is the following.
Let $\calF_1(G)$ denote the flow polytope of unit flows on $G$.

\begin{theorem}\label{thm:volumeequalsmatching}
For $H$ a connected bipartite graph with no vertices of degree zero or one and $G=G(H)$ the extension of $H$, the normalized volume $\vol(\calF_1(G))$ is equal to the number of matchings in $W(H)$.
\end{theorem}

\begin{proof}
Corollary~\ref{cor:cliquematchingbijection} establishes a bijection between the maximal simplices in a unimodular triangulation of $\calF_1(G)$ and the matchings in $W(H)$.
The result then follows by applying Theorem~\ref{thm:volumeunimodular}, which states that the number of maximal simplices in a unimodular triangulation equals the normalized volume of the polytope.
\end{proof}

The second main result of this work is the following refinement of Theorem~\ref{thm:volumeequalsmatching}.
The proof of Theorem~\ref{thm:matchingequalshstar} is given at the end of Section~\ref{sec:ehrhartbijection}.
We denote by $\mu(W(H);z)$ the unsigned matching polynomial of $W(H)$.

\begin{theorem}\label{thm:matchingequalshstar}
For $H$ a connected bipartite graph with no vertices of degree zero or one and $G=G(H)$ the extension of $H$, we have
\begin{align*}
    \mu(W(H);z) = h^*(\calF_1(G);z)\ .
\end{align*}
\end{theorem}

We obtain from Theorem~\ref{thm:matchingequalshstar} the following corollary.

\begin{corollary}\label{cor:hstarrealrooted}
For $H$ a connected bipartite graph with no vertices of degree zero or one and $G=G(H)$ the extension of $H$, the Ehrhart $h^*$-polynomial $h^*(\calF_1(G);z)$ is real-rooted, hence log-concave and unimodal.
\end{corollary}

\begin{proof}
    This follows directly from Theorem~\ref{thm:matchingpolyrealrooted} below.
\end{proof}

The remainder of the paper is structured as follows.
In Section~\ref{sec:background}, we review necessary background regarding flow polytopes, DKK triangulations, Ehrhart theory, and matching polynomials.
In Section~\ref{sec:volumebijection}, we establish the bijections needed to prove Theorem~\ref{thm:volumeequalsmatching}.
In Section~\ref{sec:ehrhartbijection}, we prove Theorem~\ref{thm:matchingequalshstar}.

\section{Background}\label{sec:background}

\subsection{Flow polytopes}
In this work, we are interested in the structure of the flow polytope associated with $G=G(H)$ where $H$ is a bipartite DAG.
We review here required background regarding flow polytopes found in Danilov, Karzanov and Koshevoy in~\cite{DKK12}, generally following their notation.

A \emph{route} in $G$ is a path in $G$ from $s$ to $t$, and the set of all routes is denoted $\calP=\calP(G)$. 
If $v$ is an inner vertex on a route $R$, then $R$ splits into two subpaths at $v$. 
Let $Rv$ denote the subpath of $R$ from $s$ to $v$, and let $vR$ denote the subpath from $v$ to $t$. 
Given a linear order $e_1$, $e_2$, $\ldots$, $e_{|E|}$ on the edges in $G$, for a route $R \in \calP(G)$ we define its characteristic vector to be 
\[
\bv_R :=\sum_{e_i\in R} \be_{i} \, .
\]
If the edges in $G$ are not assigned a linear order, we instead index a basis for the space $\R^{|E|}$ by the edges $e\in E(G)$.

\begin{example}
Consider the graph $G(K_{3,2})$ in Figure~\ref{fig:G(H)labels}.
The edges $\alpha_{2,2}\beta_{2,5}\gamma_{5,2}$ form a route in $G$, which can be split at vertex $2$ and vertex $5$.
The characteristic vector for this route is $\be_{\alpha_{2,2}}+\be_{\beta_{2,5}}+\be_{\gamma_{5,2}}\in \R^{16}$.
\end{example}

\begin{definition}\label{def:flow}
A \emph{flow} $\flow$ on a DAG $G$ is a function $\flow: E(G) \to \R$ which preserves flow at each inner vertex, i.e., for every inner vertex $v$ we have
\[
\sum_{e\,\in \, \mathrm{in}(v)} \flow(e) = \sum_{e\,\in \, \mathrm{out}(v)} \flow(e) \, .
\]
Let $\calF = \calF(G)$ denote the space of flows on $G$, and let $\calF_+ = \calF_+(G)$ denote the cone of flows satisfying $x_e\geq 0$ for all edges $e\in E(G)$.
The \emph{flow polytope} $\calF_1 = \calF_1(G)$ is the set of all nonnegative flows on $G$ of size one, i.e., flows satisfying
\[
\sum_{\substack{v \text{ is a source} \\ e \,\in \, \mathrm{out}(v)}} \flow(e) = 1 \, .
\]
\end{definition}

The flow polytope $\calF_1$ is defined as an intersection of halfspaces and a hyperplane; the total unimodularity of the constraint matrix implies the following well-known description of the vertices.

\begin{proposition}\label{prop:flowbasics}
The vertices of $\calF_1$ are the characteristic vectors $\{\bv_R: R\in \calP(G)\}$ and 
\[
\dim( \mathcal{F}_1) = |E(G)| - \#\{v \in V(G) : v \text{ is an inner vertex } \} -1 \, .
\]
\end{proposition}

The facet-supporting hyperplanes of $\calF_1(G)$ are given by $x_e\geq 0$ for the edges $e\in E(G)$.
An edge is called \emph{idle} if it is the only incoming or outgoing edge from an inner vertex. 
Contracting an idle edge $e$ corresponds geometrically to a projection along the coordinate $x_e$, and we can find non-redundant facet descriptions of the flow polytope by contracting all idle edges.

\begin{definition}\label{def:reductiontofull}
Let $G = G^0$ be a DAG with $n$ idle edges. 
For $i=1,\dots,n$, construct a graph $G^i$ with $n-i$ idle edges by contracting an idle edge of $G^{i-1}$.
The unique resulting DAG $G^n$ is the \emph{complete contraction} of $G$.
\end{definition}

\begin{proposition}\label{prop:contractidle}
If $G$ is a DAG, then the set of facets of the flow polytope $\calF_1(G)$ can be identified with the set of edges in a complete contraction of $G$.
\end{proposition}

Note that in this work we consider $G(H)$ for $H$ bipartite without any vertices of degree zero or one; hence, the $G(H)$ we consider will never have idle edges and will automatically be a complete contraction.

\subsection{DKK Triangulations}
Recall that a \emph{triangulation} of a lattice polytope $P$ is a set of lattice simplices that forms a simplicial complex and whose union is $P$.
Danilov, Karzanov and Koshevoy~\cite{DKK12} identified a family of regular unimodular triangulations of a flow polytope based on the combinatorial method of framing.

\begin{definition}
A \emph{framing} $F$ at an inner vertex $v$ of $G$ is a specification of linear orders $\preceq_{F,\inedge(v)}$ and $\preceq_{F,\outedge(v)}$ on the set of edges incoming to $v$ and on the set of edges outgoing from $v$.  
A \emph{framing} of a DAG $G$ is defined by choosing a framing at every inner vertex of $G$.
A \emph{framed graph} $(G,F)$ is a pair consisting of a DAG $G$ with a specific framing $F$ of $G$.
\end{definition}

Figure~\ref{fig:framedcompletebipartiteextension} provides an example of a framing on $G(K_{3,2})$.
In the case of an extension of a bipartite graph, we will rely on the following canonical framing.

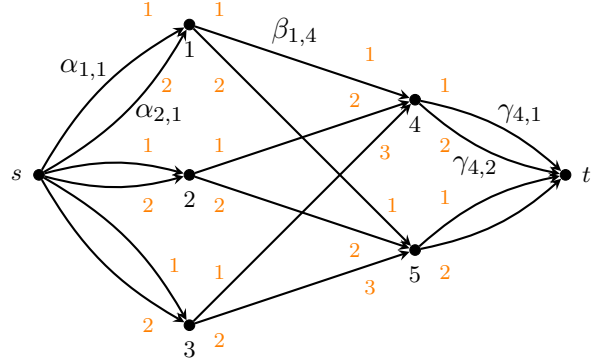
\begin{figure}
\begin{center}
\begin{tikzpicture}[scale=2]
\vertex[fill,label=below:\scriptsize{$1$}](a1) at (0,2) {};\vertex[fill,label=below:\scriptsize{$2$}](a2) at (0,1) {};
\vertex[fill,label=below:\scriptsize{$3$}](a3) at (0,0) {};
\vertex[fill,label=below:\scriptsize{$4$}](a4) at (1.5,1.5) {};
\vertex[fill,label=below:\scriptsize{$5$}](a5) at (1.5,0.5) {};
\vertex[fill,label=left:\scriptsize{$s$}](s) at (-1,1) {};
\vertex[fill,label=right:\scriptsize{$t$}](t) at (2.5,1) {};

\draw[-stealth, thick] (a1)--(a4);
\draw[-stealth, thick] (a1)--(a5);
\draw[-stealth, thick] (a2)--(a4);
\draw[-stealth, thick] (a2)--(a5);
\draw[-stealth, thick] (a3)--(a4);
\draw[-stealth, thick] (a3)--(a5);

\draw[-stealth, thick] (s) to[bend left=15] (a1);
\draw[-stealth, thick] (s) to[bend right=15] (a1);
\draw[-stealth, thick] (s) to[bend left=15] (a2);
\draw[-stealth, thick] (s) to[bend right=15] (a2);
\draw[-stealth, thick] (s) to[bend left=15] (a3);
\draw[-stealth, thick] (s) to[bend right=15] (a3);
\draw[-stealth, thick] (a4) to[bend left=15] (t);
\draw[-stealth, thick] (a4) to[bend right=15] (t);
\draw[-stealth, thick] (a5) to[bend left=15] (t);
\draw[-stealth, thick] (a5) to[bend right=15] (t);
\draw[-stealth, thick] (t);
\draw[-stealth, thick] (t);

\node[orange] at (-0.275,2.1) {\tiny $1$};
\node[orange] at (-0.15,1.6) {\tiny $2$};
\node[orange] at (-0.275,1.2) {\tiny $1$};
\node[orange] at (-0.275,.8) {\tiny $2$};
\node[orange] at (-.1,.4) {\tiny $1$};
\node[orange] at (-0.275,0) {\tiny $2$};
\node[orange] at (1.7,1.6) {\tiny $1$};
\node[orange] at (1.7,1.2) {\tiny $2$};
\node[orange] at (1.7,.85) {\tiny $1$};
\node[orange] at (1.7,.35) {\tiny $2$};

\node[orange] at (0.2,2.1) {\tiny $1$};
\node[orange] at (0.2,1.6) {\tiny $2$};
\node[orange] at (0.2,1.2) {\tiny $1$};
\node[orange] at (0.2,.8) {\tiny $2$};
\node[orange] at (0.2,.35) {\tiny $1$};
\node[orange] at (0.2,-.1) {\tiny $2$};
\node[orange] at (1.2,1.8) {\tiny $1$};
\node[orange] at (1.1,1.5) {\tiny $2$};
\node[orange] at (1.3,1.15) {\tiny $3$};
\node[orange] at (1.35,.8) {\tiny $1$};
\node[orange] at (1.1,.5) {\tiny $2$};
\node[orange] at (1.2,.25) {\tiny $3$};

\node at (-0.7,1.7) {\small $\alpha_{1,1}$};
\node at (-0.2,1.4) {\small $\alpha_{2,1}$};
\node at (0.7,1.95) {\small $\beta_{1,4}$};
\node at (2.2,1.4) {\small $\gamma_{4,1}$};
\node at (1.9,1.05) {\small $\gamma_{4,2}$};
\end{tikzpicture}
\end{center}
    \caption{An example of a framed $G(K_{3,2})$. Note that five of the edges have been labeled by their $\alpha$, $\beta$, $\gamma$ notation. Note also that the framing labels agree with the orderings induced by the canonical bipartite framing.}
\label{fig:framedcompletebipartiteextension}
\end{figure}

\begin{definition}
    Let $H$ be a bipartite graph with no vertices of degree zero or one, and let $G(H)$ be its extension graph.
    We define the \emph{canonical bipartite framing} to be the framing of $G(H)$ where at each vertex $i$:
    \begin{itemize}
        \item 
    $\alpha_{1,i}<\alpha_{2,i}$,
    \item $\gamma_{i,1}<\gamma_{i,2}$,
    \item $\beta_{i,j}<\beta_{i,k}$ if $j<k$, and
    \item $\beta_{j,i}<\beta_{k,i}$ if $j<k$.
    \end{itemize}
Further, when drawing $G(H)$, we order the vertices of each shore of $G(H)$ in increasing order from top to bottom, and we list both the $\alpha$ and $\gamma$ edges in increasing order from top to bottom.
\end{definition}

\begin{example}
    Note that the framing of $G(K_{3,2})$ given in Figure~\ref{fig:framedcompletebipartiteextension} agrees with the canonical bipartite framing.
    For example, $\beta_{1,4}<\beta_{1,5}$ is denoted in the figure by labeling at vertex $1$ the edge $\beta_{1,4}$ with $1$ and the edge $\beta_{1,5}$ with $2$ to indicate their relative order, since $4<5$.
    Similarly, $\gamma_{4,1}<\gamma_{4,2}$ is indicated by the labeling of those edges by $1$ and $2$, respectively.
\end{example}

For $v$ an inner vertex of $G$, let $\In(v)$ denote the set of partial routes in $G$ from a source vertex to $v$.
Let $\Out(v)$ denote the set of partial routes in $G$ from $v$ to a sink vertex.
A framing $F$ of $G$ induces a total order on these sets of partial routes as follows.

\begin{definition}
For $v$ an inner vertex of $G$, let $R, S \in \In(v)$ be partial routes which agree on the subroutes $R'\subseteq R$ and $S'\subseteq S$ from vertex $u$ to $v$ such that the edges immediately preceding $u$ on $R$ and on $S$ are distinct.
We denote these preceding edges by $e_R$ and $e_S$, respectively.
We define $R \preceq_{F,\In(v)} S$ if and only if $e_R \preceq_{F, \inedge(u)} e_S$, and we similarly define $\preceq_{F,\Out(v)}$.
\end{definition}

The key ingredient for framing triangulations is the notion of coherence among routes.

\begin{definition}
Let $R$ and $S$ be routes of $G$ containing a common subroute from $u$ to $v$, where $u$ and $v$ might be the same vertex, which we denote by $[u,v]$.
The routes $R$ and $S$ are said to be \emph{in conflict}, or \emph{incoherent, at $[u,v]$} if $Ru$ and $Su$ are ordered differently from $vR$ and $vS$.  
Otherwise, we say they are \emph{coherent at $[u,v]$}.
The routes $R$ and $S$ are \emph{coherent} if they are coherent at every shared subroute.
A route is \emph{exceptional} if it is coherent with every route of $G$.
A \emph{clique} of $G$ is a collection of pairwise coherent routes with respect to the framing $F$.    
\end{definition}

\begin{example}
    In Figure~\ref{fig:framedcompletebipartiteextension}, the routes $\alpha_{2,1}\beta_{1,5}\gamma_{5,2}$ and $\alpha_{1,1}\beta_{1,5}\gamma_{5,1}$ are coherent while the routes $\alpha_{2,1}\beta_{1,5}\gamma_{5,1}$ and $\alpha_{1,1}\beta_{1,5}\gamma_{5,2}$ are in conflict.
\end{example}

This definition of coherence leads to the following theorem that establishes the \emph{DKK triangulation of $\calF_1(G)$ corresponding to the framing $F$}, denoted by $\DKK(G,F)$.
For each clique $C$ of $(G,F)$, we defined the \emph{clique simplex} to be the convex hull of the indicator vectors of the routes in $C$.
We will frequently refer to the simplex for a clique as the clique itself.

\begin{theorem}[{Danilov, Karzanov, Koshevoy~\cite{DKK12}}]\label{thm:dkk-triangulation}
Let $(G,F)$ be a framed DAG.
The set of cliques of $G$ with respect to the framing $F$ forms a regular unimodular triangulation of $\calF_1(G)$.
\end{theorem}

\begin{remark}
    The \emph{coherence graph} for $(G,F)$ is the graph where vertices are the routes of $G$ and where two routes form an edge if and only if the routes are coherent. 
    As a simplicial complex, the DKK triangulation is isomorphic to the clique complex of the coherence graph, hence the use of ``clique'' to describe a pairwise coherent set of routes. 
\end{remark}

\subsection{Ehrhart Theory}\label{sec:ehrhart}

A $d$-dimensional convex \emph{lattice polytope} $P\subseteq\R^n$ is the convex hull of finitely many points in $\Z^n$ that span a $d$-dimensional affine subspace of $\R^n$.
We say $P$ is a \emph{$d$-simplex} if $P$ is $d$-dimensional and has $d+1$ vertices.
The \emph{Ehrhart polynomial} of $P$ is the function $i(P;t):=|tP\cap\Z^n|$, where $tP:=\{t\bp:\bp\in P\}$ denotes the $t^{th}$ dilate of the polytope $P$.
The \emph{Ehrhart series} of $P$ is the generating function
\[
  E(P;z) := \sum_{t\geq0}i(P;t)z^t = \frac{h_0^*+h_1^*z+\cdots+h_d^*z^d}{(1-z)^{\dim(P)+1}} \, .
\]
The fact that this is a rational function is due to Ehrhart~\cite{ehrhart}.
The coefficients $h^\ast_0,h^\ast_1,\ldots,h^\ast_d$ are known to be nonnegative integers~\cite{StanleyDecompositions}.
The polynomial $h^\ast(P;z):=h_0^*+h_1^*z+\cdots+h_d^*z^d$ is called the \emph{(Ehrhart) $h^\ast$-polynomial} of $P$.

An important property of a lattice polytope $P$ is its volume; the \emph{normalized volume} of $P$ is defined to be $d!$ times the Euclidean volume of $P$ with respect to the integer lattice induced by the affine span of $P$.
It is also known that
\[
\vol(P):=\sum_jh^*_j \, ;
\]
for a textbook treatment, see Beck and Robins~\cite{BeckRobinsCCD}.
Further, the following theorem connects the normalized volume of $P$ with unimodular triangulations.

\begin{theorem}\label{thm:volumeunimodular}
    If a lattice polytope $P$ admits a unimodular triangulation $\calT$, then $\vol(P)$ is equal to the number of maximal simplices in $\calT$.
\end{theorem}

\subsection{Matching Polynomials}

Let $G$ be a graph with $n$ vertices and let $m_k$ be the number of $k$-edge matchings of $G$. 
The \emph{(unsigned) matching polynomial} of $G$, denoted $\mu(G;z)$, is defined as
\begin{align*}
    \mu(G;z) := \sum_{k=0}^{\floor{n/2}} m_k z^k\ .
\end{align*}

A foundational result regarding matching polynomials is the following, due to Heilmann and Lieb~\cite{heilmannlieb,stanleylogconcave}.

\begin{theorem}[Heilmann and Lieb~\cite{heilmannlieb}]
\label{thm:matchingpolyrealrooted}
    All the roots of the matching polynomial $\mu(G;z)$ are  real.
    As a corollary, the coefficient sequence of $\mu(G;z)$ is both log-concave and unimodal.
\end{theorem}

\section{Volumes via Matchings}\label{sec:volumebijection}

In this section, we establish a bijection between cliques of routes in $G(H)$ for the canonical bipartite framing, matchings in $W(H)$, and a set of vectors that we refer to as clique vectors.
This bijection is the key to proving Theorem~\ref{thm:volumeequalsmatching}, and is the basis for the proof of Theorem~\ref{thm:matchingequalshstar} given in Section~\ref{sec:ehrhartbijection}.

We begin by setting notation for our bipartite graphs.
Assume that $H=S\uplus T$ is a bipartite graph with no vertex of degree zero or one, where $S=\{1,2,\ldots,n\}$ and $T=\{n+1,n+2,\ldots,n+m\}$.
Let $N_H(v)$ denote the set of neighbors of $v$ in $H$.

\begin{definition}
    Given $H$ as above, let $\cliques(H)$ denote the set of maximal cliques of $G(H)$ with respect to the canonical bipartite framing and let $\calM(H)$ denote the set of matchings in $W(H)$.
\end{definition}

\begin{definition}\label{def:cliquevector}
Given $H$ as above, a \emph{clique vector} for $H$ is a vector 
\[
\ba\in \{n+1,n+2,\ldots,n+m\}^{n}\times \{1,2,\ldots,n\}^{m}\times \{+,-\}^{E(H)}
\]
satisfying the following properties:
\begin{itemize}
    \item $a_i\in N_H(i)$ for each $i\in \{1,2,\ldots,n+m\}$, and
    \item $a_{i,j} \in \{+,-\}$ when both $a_i=j$ and $a_j=i$, and otherwise $a_{i,j} = -$.
\end{itemize}
When writing a clique vector, we list the entries $a_{i,j}$ lexicographically.
We denote by $\cliquevs(H)$ the set of all clique vectors of $H$.
\end{definition}

\begin{example}\label{ex:k32cliquevectors}
Consider $K_{3,2}$ as shown in Figure~\ref{fig:G(H)labels}.
Observe that
\begin{align*}
\ba&=(a_1,a_2,a_3,a_4,a_5,a_{1,4},a_{1,5},a_{2,4},a_{2,5},a_{3,4},a_{3,5}) \\
&= (4,5,4,1,3,-,-,-,-,-,-)
\end{align*}
is a valid clique vector, since each $a_i\in N_H(i)$ and ``$-$'' is always allowed as an edge-indexed entry. 
Further, we have that
\begin{align*}
\ba&= (a_1,a_2,a_3,a_4,a_5,a_{1,4},a_{1,5},a_{2,4},a_{2,5},a_{3,4},a_{3,5}) \\
&= (4,5,4,1,3,+,-,-,-,-,-)
\end{align*}
is a valid clique vector, since $a_1=4$ and $a_4=1$, thus $a_{1,4}=+$ is allowed.
However,
\[
\ba=(4,5,4,1,3,-,-,-,-,+,-)
\]
is not a valid clique vector, since $a_3=4$ while $a_4=1$.
\end{example}

We next define a map $\phi$ that sends each clique vector of $H$ to a maximal clique of $G(H)$ with respect to the canonical bipartite framing of $G(H)$.
For each clique vector $\ba$, the clique $\phi(\ba)$ will include for each edge $ij\in E(H)$ some number of routes passing through $\beta_{i,j}\in E(G(H))$, where the number of routes will fall into one of three cases.
The first case occurs for an edge $\beta_{i,j}$ where $a_i\neq j$ and $a_j\neq i$, in which case one route through $\beta_{i,j}$ is included in $\phi(\ba)$. 
The second case is when either $a_i=j$ or $a_j=i$, but not both, in which case two routes are included.
The third case is when both $a_i=j$ and $a_j=i$; in this case, there will be three routes contributed, with two subcases depending on the value of $a_{i,j}\in \{+,-\}$.

\begin{definition}\label{def:cliquemap}
    We define the \emph{clique map} $\phi:\cliquevs(H)\to \cliques(H)$ as follows.
    For each edge $ij\in E(H)$, we include in $\phi(\ba)$ the following routes:
    \begin{enumerate}
    \item[(i)] $a_i\neq j$ and $a_j\neq i$:
        \begin{itemize}
        \item If $a_i<j$ and $a_j<i$, include the route $\alpha_{2,i}\beta_{i,j}\gamma_{j,2}$.
        \item If $a_i<j$ and $a_j>i$, include the route $\alpha_{2,i}\beta_{i,j}\gamma_{j,1}$.
        \item If $a_i>j$ and $a_j<i$, include the route $\alpha_{1,i}\beta_{i,j}\gamma_{j,2}$.
        \item If $a_i>j$ and $a_j>i$, include the route $\alpha_{1,i}\beta_{i,j}\gamma_{j,1}$.
        \end{itemize}
        \item[(ii)] Either $a_i=j$ or $a_j=i$, but not both:
        \begin{itemize}
        \item If $a_i=j$ and $a_j<i$, include the routes $\alpha_{1,i}\beta_{i,j}\gamma_{j,2}$ and $\alpha_{2,i}\beta_{i,j}\gamma_{j,2}$.
        \item If $a_i=j$ and $a_j>i$, include the routes $\alpha_{1,i}\beta_{i,j}\gamma_{j,1}$ and $\alpha_{2,i}\beta_{i,j}\gamma_{j,1}$.
        \item If $a_j=i$ and $a_i<j$, include the routes $\alpha_{2,i}\beta_{i,j}\gamma_{j,1}$ and $\alpha_{2,i}\beta_{i,j}\gamma_{j,2}$.
        \item If $a_j=i$ and $a_i>j$, include the routes $\alpha_{1,i}\beta_{i,j}\gamma_{j,1}$ and $\alpha_{1,i}\beta_{i,j}\gamma_{j,2}$.
        \end{itemize}
        \item[(iii)] $a_i=j$ and $a_j=i$:
    \begin{itemize}
        \item If $a_{i,j}=-$, include the routes $\alpha_{1,i}\beta_{i,j}\gamma_{j,1}$, $\alpha_{1,i}\beta_{i,j}\gamma_{j,2}$, and $\alpha_{2,i}\beta_{i,j}\gamma_{j,2}$.
        \item If $a_{i,j}=+$, include the routes $\alpha_{1,i}\beta_{i,j}\gamma_{j,1}$, $\alpha_{2,i}\beta_{i,j}\gamma_{j,1}$, and $\alpha_{2,i}\beta_{i,j}\gamma_{j,2}$.
        \end{itemize}
    \end{enumerate}
\end{definition}

\begin{example}\label{ex:cliquemap}
    Consider $G(K_{3,2})$ with the canonical bipartite framing as shown in Figure~\ref{fig:framedcompletebipartiteextension}, and consider the clique vector of $K_{3,2}$ given by
\[
\ba=(4,5,4,1,3,+,-,-,-,-,-) \, .
\]
We determine six of the routes in $\phi(\ba)$, which are all shown in Figure~\ref{fig:cliquemapex}.
Note that these are only six of the eleven routes in $\phi(\ba)$, as the full clique contains five additional routes beyond those shown.
The edge $24$ has $a_2=5$ and $a_4=1$.
Thus, we are in case (i) of the definition of $\phi$ with $a_2=5>4$ and $a_4=1<2$. 
This falls within the third subcase of (i), hence $\phi(\ba)$ contains the route $\alpha_{1,2}\beta_{2,4}\gamma_{4,2}$.
The edge $35$ has $a_3=4$ and $a_5=3$, falling into case (ii) of the definition with $a_3=4<5$ and $a_5=3$.
So, this falls within the third subcase of (ii), hence $\phi(\ba)$ contains the routes $\alpha_{2,3}\beta_{3,5}\gamma_{5,1}$ and $\alpha_{2,3}\beta_{3,5}\gamma_{5,2}$.
Finally, the edge $14$ has $a_1=4$ and $a_4=1$ with $a_{1,4}=+$, hence falls within the second subcase of (iii).
Thus, $\phi(\ba)$ contains the routes $\alpha_{1,1}\beta_{1,4}\gamma_{4,1}$, $\alpha_{2,1}\beta_{1,4}\gamma_{4,1}$, and $\alpha_{2,1}\beta_{1,4}\gamma_{4,2}$.

Note that if we had $a_{1,4}=-$ in the clique vector $\ba$, then instead of $\alpha_{2,1}\beta_{1,4}\gamma_{4,1}$, we would include the route $\alpha_{1,1}\beta_{1,4}\gamma_{4,2}$.
Note that the route $\alpha_{1,1}\beta_{1,4}\gamma_{4,1}$ will always be ``above'' the other two routes on $\beta_{1,4}$ in our drawing, while $\alpha_{2,1}\beta_{1,4}\gamma_{4,2}$ will always be ``below'' the other two routes on $\beta_{1,4}$.
Thus, the middle route can either go from a visually upper $\alpha$ edge to a visually lower $\gamma$ edge, in which case we have a ``$-$'' sign, or it can go from a visually lower $\alpha$ edge to a visually upper $\gamma$ edge, in which case we have a ``$+$''.
\end{example}

We next show that $\phi$ is a bijection.

\begin{figure}
\begin{center}
\begin{tikzpicture}[scale=2]
\vertex[fill](a1) at (0,1.95) {};\vertex[fill](a11) at (0,2.1) {};\vertex[fill,label=above:\scriptsize{$1$}](a12) at (0,2.2) {};\vertex[fill,label=below:\scriptsize{$2$}](a2) at (0,1) {};
\vertex[fill,label=below:\scriptsize{$3$}](a3) at (0,0) {};
\vertex[fill](a31) at (0,0.1) {};
\vertex[fill](a4) at (1.5,1.45) {};
\vertex[fill](a41) at (1.5,1.6) {};
\vertex[fill,label=above:\scriptsize{$4$}](a42) at (1.5,1.7) {};
\vertex[fill](a40) at (1.5,1.35) {};
\vertex[fill](a51) at (1.5,0.5) {};
\vertex[fill,label=below:\scriptsize{$5$}](a52) at (1.5,0.4) {};
\vertex[fill,label=left:\scriptsize{$s$}](s) at (-1,1) {};
\vertex[fill,label=right:\scriptsize{$t$}](t) at (2.5,1) {};

\draw[-stealth, thick,densely dashdotdotted] (a1)--(a4);
\draw[-stealth, thick] (a11)--(a41);
\draw[-stealth, thick,dashdotted] (a12)--(a42);
\draw[-stealth, thick,dotted] (a2)--(a40);
\draw[-stealth, thick] (a31)--(a51);
\draw[-stealth, thick,dashed] (a3)--(a52);

\draw[-stealth, thick,dashdotted] (s) to[bend left=30] (a12);
\draw[-stealth, thick] (s) to[bend right=20] (a11);
\draw[-stealth, thick,densely dashdotdotted] (s) to[bend right=30] (a1);
\draw[-stealth, thick,dotted] (s) to[bend left=25] (a2);
\draw[-stealth, thick] (s) to[bend right=35] (a31);
\draw[-stealth, thick,dashed] (s) to[bend right=35] (a3);
\draw[-stealth, thick] (a41) to[bend left=30] (t);
\draw[-stealth, thick,dashdotted] (a42) to[bend left=35] (t);
\draw[-stealth, thick,densely dashdotdotted] (a4) to[bend right=25] (t);
\draw[-stealth, thick] (a51) to[bend left=25] (t);
\draw[-stealth, thick,dotted] (a40) to[bend right=30] (t);
\draw[-stealth, thick,dashed] (a52) to[bend right=35] (t);
\end{tikzpicture}
\end{center}
\caption{The routes from Example~\ref{ex:cliquemap}. For clarity, edges and inner vertices have been duplicated when multiple routes pass through them. Thus, for example, edge $\beta_{3,5}$ is shown twice, once as a solid line and once as a dashed line.}
\label{fig:cliquemapex}
\end{figure}
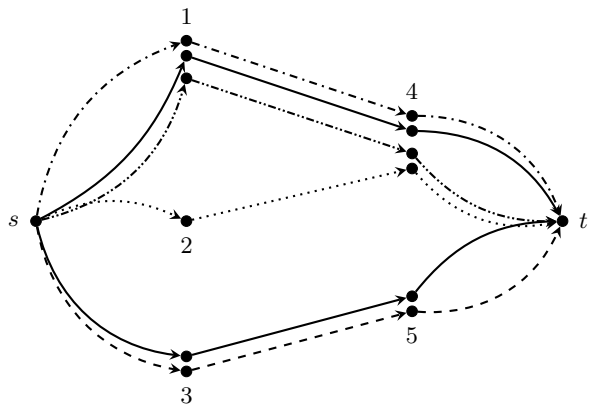

\begin{theorem}\label{thm:cliquebijection}
Given a bipartite graph $H$ with no vertex of degree zero or one, the map $\phi$ is a well-defined bijection.
\end{theorem}

\begin{proof}
We begin by showing that $\phi$ is well-defined.
Fix a clique vector $\ba$.
We first show that $\phi(\ba)$ has maximal cardinality.
From Proposition~\ref{prop:flowbasics}, it follows that every maximal clique in $G(H)$ contains $n+m+|E(H)|$ routes.
Note that each $\beta_{i,j}$ is contained in at least one route of $\phi(\ba)$, giving $|E(H)|$ routes.
For each vertex $i\in\{1,2,\ldots,n\}$, there is exactly one $j$ with $a_i=j$.
For that edge $ij$, if $a_j\neq i$, then $i$ contributes one additional route from condition (ii) in Definition~\ref{def:cliquemap}.
Similarly, for each $j\in \{n+1,\ldots,n+m\}$, there is exactly one $i$ with $a_j=i$ and if $a_i\neq j$, the vertex $j$ contributes one additional route from condition (ii).
Those edges $ij$ satisfying $a_i=j$ and $a_j=i$ contribute two routes from condition (iii), i.e., one route for each vertex.
Thus, in addition to one route in $\phi(\ba)$ for each $ij\in E(H)$, there is one additional route contributed by each vertex of $H$.
Thus, there are $n+m+|E(H)|$ routes in $\phi(\ba)$.

We next show that these routes are pairwise coherent with respect to the canonical bipartite framing.
Observe that there are two ways for routes to be in conflict in $G(H)$.
The first way is for the pair of routes to be, for a fixed $i$ and $j$, of the form
\[
\alpha_{1,i}\beta_{i,j}\gamma_{j,2}\,
\text{  and  }\,
\alpha_{2,i}\beta_{i,j}\gamma_{j,1} \, .
\]
By construction, at most one of these two routes will be included in $\phi(\ba)$, as they share the edge $\beta_{i,j}$ and Definition~\ref{def:cliquemap} never adds both of these two routes.
The second way for a conflict to arise is a conflict at a fixed vertex, say at vertex $i\in \{1,2,\ldots,n\}$ (the case of $j\in \{n+1,n+2,\ldots,n+m\}$ is handled by an identical argument to what follows).
In this case, we have two routes,
\[
\alpha_{1,i}\beta_{i,k}\gamma_{k,c_1} \quad \text{ and } \quad 
\alpha_{2,i}\beta_{i,j}\gamma_{j,c_2},
\]
with $j<k$.
However, $\alpha_{1,i}$ appears in the first route only when $k\leq a_i$, and $\alpha_{2,i}$ appears in the second route only when $a_i\leq j$.
Combining these two inequalities, we have $k\leq a_i\leq j$, which contradicts the assumption that $j<k$.
Thus, again by the definition of $\phi$, this pair of conflicting routes will never appear in $\phi(\ba)$.

Having established that $\phi$ is well-defined, we next show that $\phi$ is injective.
Suppose for a contradiction that $\ba\neq \ba'$ are two clique vectors; we will show that $\phi(\ba)\neq \phi(\ba')$.
If $a_i=a'_i$ for every $i$ but $a_{i,j}\neq a'_{i,j}$ for some $ij$, then the three routes through $\beta_{i,j}$ constructed in condition (iii) of Definition~\ref{def:cliquemap} would be different, and thus $\phi(\ba)\neq \phi(\ba')$.
If there exists an index $i$ such that $a_i\neq a'_i$ with $i\in \{1,\ldots,n\}$ (the case of $j\in \{n+1,\ldots,n+m\}$ is similar), then in $\phi(\ba)$ there would be routes starting with both $\alpha_{1,i}\beta_{i,a_i}$ and $\alpha_{2,i}\beta_{i,a_i}$, while in $\phi(\ba')$ any routes on $\beta_{i,a_i}$ would either start with $\alpha_{1,i}$ or $\alpha_{2,i}$, but not both.
Thus, $\phi(\ba)\neq \phi(\ba')$, and hence $\phi$ is injective.

To conclude the proof, we need to show that $\phi$ is surjective.
Observe that there are exactly four routes in $G(H)$ passing through $\beta_{i,j}$, which are
\begin{equation}\label{eq:fourroutes}
\alpha_{1,i}\beta_{i,j}\gamma_{j,1} \, , \alpha_{2,i}\beta_{i,j}\gamma_{j,1} \, , \alpha_{1,i}\beta_{i,j}\gamma_{j,2} \, , \text{ and }\alpha_{2,i}\beta_{i,j}\gamma_{j,2} \, .
\end{equation}
Note that the middle two routes are in conflict for the canonical bipartite framing, and thus cannot appear together in a clique.

Let $C\in \cliques(H)$.
We first prove by contradiction that for each edge $ij\in E(H)$, there must be at least one route in $C$ passing through $\beta_{i,j}$.
Suppose not.
If no other edges of $H$ adjacent to $i$ or $j$ are used in routes, then any route through $\beta_{i,j}$ can be added to $C$, contradicting that $C$ is a maximal clique.
Suppose that $\beta_{i,\ell}$ and/or $\beta_{i,g}$ are edges in routes of $C$ such that $\ell<j<g$.
Then the routes in $C$ cannot contain both $\alpha_{2,i}\beta_{i,\ell}$ and $\alpha_{1,i}\beta_{i,g}$, as that would cause a conflict at $i$.
If $\alpha_{h,i}\beta_{i,\ell}$ and/or $\alpha_{h,i}\beta_{i,g}$ are contained in routes in $C$, we can form a partial route starting with $\alpha_{h,i}\beta_{i,j}$.
If $\alpha_{1,i}\beta_{i,\ell}$ and/or $\alpha_{2,i}\beta_{i,g}$ are contained in routes in $C$, we can form a partial route starting with $\alpha_{1,i}\beta_{i,j}$.
A similar analysis of the vertex $j$ shows that we can complete this partial route with an appropriate choice of $\gamma_{j,b}$ without creating any conflicts with other routes of $C$, contradicting maximality of $C$.

Second, we claim that for every vertex $i\in \{1,\ldots,n\}$, there is a unique edge $\beta_{i,j_0}$ that is used in routes in $C$ containing both $\alpha_{1,i}$ and $\alpha_{2,i}$ (the same is true for $j\in\{n+1,\ldots,n+m\}$ with $\gamma_{j,1}$ and $\gamma_{j,2}$, by an identical argument).
We have established that if $N_H(i)=\{j_1,\ldots,j_k\}$, then every edge $\beta_{i,j_\ell}$ is contained in at least one route in $C$.
Further, we know that if $j_\ell<j_p$, then we cannot have both 
\[
\alpha_{1,i}\beta_{i,j_p} \, \text{ and } \, \alpha_{2,i}\beta_{i,j_\ell}
\]
contained in routes in $C$, as they conflict at $i$.
Thus, there exists some $j_0$ such that for $j_\ell<j_0$, we have $\alpha_{1,i}\beta_{i,j_\ell}$ is contained in a route of $C$, and for $j_0<j_p$, we have $\alpha_{2,i}\beta_{i,j_p}$ is contained in a route of $C$; our goal is to show that both $\alpha_{1,j_0}\beta_{i,j_0}$ and $\alpha_{2,j_0}\beta_{i,j_0}$ are contained in routes of $C$.
Suppose $\beta_{i,j_0}$ is contained in only one route in $C$, given without loss of generality by $\alpha_{1,i}\beta_{i,j_0}\gamma_{j_0,h}$ for a fixed $h\in \{1,2\}$.
Then we can also add $\alpha_{2,i}\beta_{i,j_0}\gamma_{j_0,h}$ to $C$ without causing any conflict, contradicting maximality of $C$.

Our goal now is to define a clique vector $\ba$ from $C$ such that $\phi(\ba)=C$.
We first consider any edge $\beta_{i,j}$ such that three routes on this edge are in $C$.
In this case, we define $a_i=j$ and $a_j=i$ and we set $a_{i,j}$ equal to ``$+$'' or ``$-$'' depending on which of the middle two routes in~\eqref{eq:fourroutes} are contained in $C$.
For any edge $\beta_{i,j}$ where we do not have three routes on the edge in $C$, we define $a_{i,j}=-$.
Suppose $\beta_{i,j}$ has two routes through it in $C$. 
In this case, either $\alpha_{1,i}$ and $\alpha_{2,i}$ appear in the routes, or else $\gamma_{j,1}$ and $\gamma_{j,2}$ appear in the routes, but not both.
If it is the former case, set $a_i=j$; if the latter, set $a_j=i$.

Having defined $\ba$ in this manner, it is true by construction that for any edge $\beta_{i,j}$ supporting three routes in $C$, the routes in $C$ are the ones defined by $\phi(\ba)$.
For any edge $\beta_{i,j}$ supporting two routes in $C$, it must be the case that those two routes either have the same $\alpha$ edges or the same $\gamma$ edges, as otherwise the routes would either be in conflict or would support a third coherent route.
Further, whether it is the $\alpha$ edges or the $\gamma$ edges that are the same is in agreement for $\phi(\ba)$ and $C$ by construction.
Thus, what remains is to verify that the edges $\beta_{i,j}$ supporting a single route in $C$ and in $\phi(\ba)$ are the same.
Since the triple- and double-supporting $\beta_{i,j}$ edges determine the entries of $\ba$, all the remaining routes in $\phi(\ba)$ are determined by these.
Similarly, in $C$, once the double- and triple-supported edges have their routes specified, this determines the remaining routes, as the choices of $\alpha$ and $\gamma$ edges in the remaining routes must result in coherence in the clique.
Hence, $\phi(\ba)$ and $C$ are the same clique, and the proof is complete.
\end{proof}

Our next goal is to prove that clique vectors are in bijection with matchings of $W(H)$.
We first set notation for the vertices in $W(H)$ that are not already present in $H$.
Suppose $H$ has a given bipartition $S\uplus T$, and suppose the neighbors of a vertex $i$ are $\{v_1<v_2<\cdots <v_k\}$.
We label each new neighbor of $i$ in $W(H)$ with a unique element of the set $\{w_{i,v_2},\dots,w_{i,v_k}\}$ if $i \in S$ and $\{w_{i,v_1},\dots,w_{i,v_{k-1}}\}$ if $i \in T$.

\begin{definition}\label{def:matchingmap}
    Let $\calM(H)$ denote the set of matchings of the almost-degree-whiskered graph $W(H)$. 
    We define the \emph{matching map} $\psi: \calV(H) \to \calM(H)$ as follows.
    Given $\ba\in \calV(H)$, for each vertex $i$ of $H$,
    \begin{enumerate}
        
        \item if $i \in \{1,\dots,n\}$, and if $a_i = j$ and $a_j = k \neq i$, then include the edge $iw_{i,j}$ (if it exists). 
        
        \item if $j \in \{n+1,\dots,n+m\}$, and if $a_j = i$ and $a_i = k \neq j$, then include the edge $jw_{j,i}$ (if it exists).
        
        \item if $a_i = j$, $a_j = i$, and $a_{i,j} = -$, then include the edge $ij$.
        
        \item if $a_i = j$, $a_j = i$, and $a_{i,j} = +$, then include the edge $iw_{i,j}$ (if it exists) and the edge $jw_{j,i}$ (if it exists). 
    \end{enumerate}    
\end{definition}

\begin{example}\label{ex:matchingmap}
    Recall from Example~\ref{ex:k32cliquevectors} that     
    \[
        \ba=(4,5,4,1,3,-,-,-,-,-,-) \text{ and } \ba' = (4,5,4,1,3,+,-,-,-,-,-)
    \]
    are both valid clique vectors of $G(K_{3,2})$.
    Figure~\ref{fig:psi example} shows the almost-degree-whiskered graph $W(K_{3,2})$ as well as the two corresponding matchings $\psi(\ba)$ and $\psi(\ba')$.

    The edge $14$ is in $\psi(\ba)$ since $a_1 = 4$, $a_4=1$, and $a_{1,4}=-$.
    Since $a_2 = 5$ but $a_5 \neq 2$, $\psi(\ba)$ also contains the edge $2w_{2,5}$.
    However, since $a_3 = 4$, $a_4 \neq 3$, and there is no vertex $w_{3,4}$, $\psi(\ba)$ will not include any edge covering $3$.
    For a similar reason, there is no edge in $\psi(\ba)$ covering $5$.

    The vector $\ba'$ differs from $\ba$ only in $a'_{14}=+$, so $\psi(\ba')$ contains the edge $4w_{4,1}$; note that $1w_{1,4}$ is not present in $W(H)$, and therefore we omit it.
\end{example}

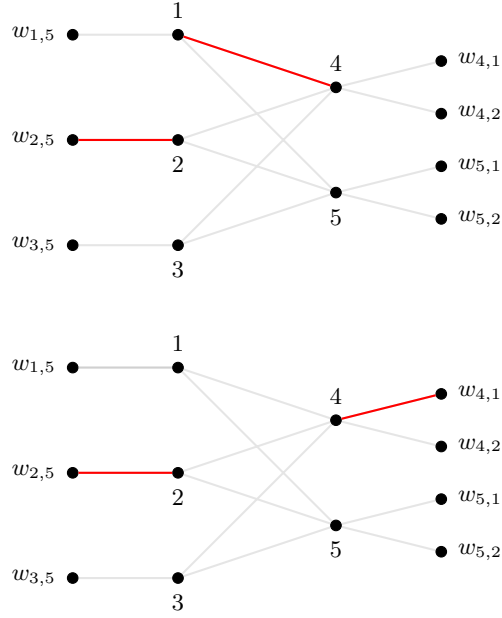
\begin{figure}
\begin{center}
    \begin{tikzpicture}[scale=1.4]
    \begin{scope}
        [yshift=90]
        \vertex[fill,label=left:\scriptsize{$w_{1,5}$}](l1) at (-1,2) {};
        \vertex[fill,label=left:\scriptsize{$w_{2,5}$}](l2) at (-1,1) {};
        \vertex[fill,label=left:\scriptsize{$w_{3,5}$}](l3) at (-1,0) {};
        \vertex[fill,label=right:\scriptsize{$w_{4,1}$}](r41) at (2.5,1.75) {};
        \vertex[fill,label=right:\scriptsize{$w_{4,2}$}](r42) at (2.5,1.25) {};
        \vertex[fill,label=right:\scriptsize{$w_{5,1}$}](r51) at (2.5,0.75) {};
        \vertex[fill,label=right:\scriptsize{$w_{5,2}$}](r52) at (2.5,0.25) {};
        
        \vertex[fill,label=above:\scriptsize{$1$}](a1) at (0,2) {};
        \vertex[fill,label=below:\scriptsize{$2$}](a2) at (0,1) {};
        \vertex[fill,label=below:\scriptsize{$3$}](a3) at (0,0) {};
        \vertex[fill,label=above:\scriptsize{$4$}](b4) at (1.5,1.5) {};
        \vertex[fill,label=below:\scriptsize{$5$}](b5) at (1.5,0.5) {};

        \draw[gray,thick,opacity=0.2] (b5)--(a1)--(b4)--(a2)--(b5)--(a3)--(b4);
        \draw[gray,thick,opacity=0.2,thick] (l1)--(a1);
        \draw[gray,thick,opacity=0.2,thick] (l2)--(a2);
        \draw[gray,thick,opacity=0.2,thick] (l3)--(a3);
        \draw[gray,thick,opacity=0.2,thick] (r41)--(b4);
        \draw[gray,thick,opacity=0.2,thick] (r42)--(b4);
        \draw[gray,thick,opacity=0.2,thick] (r51)--(b5);
        \draw[gray,thick,opacity=0.2,thick] (r52)--(b5);

        \draw[red,thick] (a1)--(b4);
        \draw[red,thick] (l2)--(a2);
    \end{scope}
        \begin{scope}
        \vertex[fill,label=left:\scriptsize{$w_{1,5}$}](l1) at (-1,2) {};
        \vertex[fill,label=left:\scriptsize{$w_{2,5}$}](l2) at (-1,1) {};
        \vertex[fill,label=left:\scriptsize{$w_{3,5}$}](l3) at (-1,0) {};
        \vertex[fill,label=right:\scriptsize{$w_{4,1}$}](r41) at (2.5,1.75) {};
        \vertex[fill,label=right:\scriptsize{$w_{4,2}$}](r42) at (2.5,1.25) {};
        \vertex[fill,label=right:\scriptsize{$w_{5,1}$}](r51) at (2.5,0.75) {};
        \vertex[fill,label=right:\scriptsize{$w_{5,2}$}](r52) at (2.5,0.25) {};
        
        \vertex[fill,label=above:\scriptsize{$1$}](a1) at (0,2) {};
        \vertex[fill,label=below:\scriptsize{$2$}](a2) at (0,1) {};
        \vertex[fill,label=below:\scriptsize{$3$}](a3) at (0,0) {};
        \vertex[fill,label=above:\scriptsize{$4$}](b4) at (1.5,1.5) {};
        \vertex[fill,label=below:\scriptsize{$5$}](b5) at (1.5,0.5) {};

        \draw[gray,thick,opacity=0.2] (b5)--(a1)--(b4)--(a2)--(b5)--(a3)--(b4);
        \draw[gray,thick,opacity=0.2,thick] (l1)--(a1);
        \draw[gray,thick,opacity=0.2,thick] (l2)--(a2);
        \draw[gray,thick,opacity=0.2,thick] (l3)--(a3);
        \draw[gray,thick,opacity=0.2,thick] (r41)--(b4);
        \draw[gray,thick,opacity=0.2,thick] (r42)--(b4);
        \draw[gray,thick,opacity=0.2,thick] (r51)--(b5);
        \draw[gray,thick,opacity=0.2,thick] (r52)--(b5);

        \draw[gray,thick,opacity=0.2] (l1)--(a1);
        \draw[red,thick] (l2)--(a2);
        \draw[red,thick] (b4)--(r41);
    \end{scope}
\end{tikzpicture}
\end{center}

\caption{Two matchings $\psi(\ba)$ and $\psi(\ba')$ of $W(K_{3,2})$.}\label{fig:psi example}
\end{figure}

\begin{theorem}\label{thm:matchingbijection}
Given a bipartite graph $H$ with no vertex of degree zero or one, the map $\psi$ is a well-defined bijection.
\end{theorem}

\begin{proof}
    First we show that $\psi$ is well-defined.
    Since $\psi(\ba)$ provides a set of edges of $W(H)$ and each vertex of the form $w_{i,v_j}$ has degree $1$, then each such vertex will have degree at most $1$ in $\psi(\ba)$.
    
    Now, let $i$ be a vertex in $H$ (and therefore also a vertex of $W(H)$).
    The fact that $a_i$ is a single value immediately implies that there is at most one $j$ which can satisfy any of the conditions.
    The degree of $i$ in $\psi(\ba)$ will be greater than $1$ if there is a single vertex $j$ satisfying
    \begin{itemize}
        \item conditions (3) and (4) in Definition~\ref{def:matchingmap}, or
        \item conditions (1) and (3) in Definition~\ref{def:matchingmap}, or
        \item conditions (2) and (3) in Definition~\ref{def:matchingmap}.
    \end{itemize}
    The conditions on $a_j$ and $a_{i,j}$ are mutually exclusive, hence no two of the three conditions can be simultaneously true.
    Thus, the degree of $i$ in $\psi(\ba)$ is also at most $1$, and $\psi(\ba)$ is indeed a matching of $W(H)$.

    Next, we show that $\psi$ is injective.
    Suppose $\ba, \ba' \in \calV(H)$ are distinct.
    If $a_{i,j} \neq a'_{i,j}$ for some $ij$, then one of the matchings $\psi(\ba)$, $\psi(\ba')$ contains $ij$ and the other does not, showing $\psi(\ba) \neq \psi(\ba')$.
    Otherwise, $a_i = j$ and $a_i' = k \neq j$ for some $i$.
    In this case, one of the following situations must hold:
    \begin{itemize}
        \item $ij \in \psi(\ba)$ and $ij \notin \psi(\ba')$ (or vice versa).
        \item $\psi(\ba)$ contains the edge $iw_{i,j}$ (if it exists), while $\psi(\ba')$ contains the edge  $iw_{i,k}$ (if it exists).
    \end{itemize}
    In the first situation, it is clear that $\psi(\ba) \neq \psi(\ba')$.
    In the second, if the edge $iw_{i,j}$ exists in $W(H)$, then it is in $\psi(\ba)$ but not $\psi(\ba')$.
    Similarly, if the edge $iw_{i,k}$ exists in $W(H)$, then it is in $\psi(\ba')$ but not $\psi(\ba)$.
    Since $j \neq k$, at least one of the edges containing $w_{i,j}$ or $w_{i,k}$ must exist.
    So, we can be sure that $\psi(\ba) \neq \psi(\ba')$.
    Therefore, $\psi$ is injective.

    Lastly, we show that $\psi$ is surjective.
    For $M \in \calM(H)$, we will construct a clique vector mapping to it by $\psi$.
    First, we determine the values of each $a_i$.
    If there is any edge of $M$ containing a vertex of the form $w_{i,j}$, then set $a_i = j$. 
    If $i \in S$ and $i$ is not covered in $M$, then set $a_i$ equal to the smallest-label neighbor of $i$.
    If $i \in T$ and $i$ is not covered in $M$, then set $a_i$ equal to the largest-label neighbor of $i$.
    If $ij$ is an edge of $H$ that is also in $M$, then set $a_i = j$ and $a_j = i$.

    Next, we determine the signs $a_{i,j}$ for each edge $ij$ in $H$.
    If $a_i = j$ and $a_j \neq i$, then set $a_{i,j} = -$.
    If $a_i = j$ and $a_j = i$, then check whether $ij$ is in $M$.
    If so, set $a_{i,j} = -$, and if not, then set $a_{i,j} = +$.
    Set all remaining $a_{i,j}$ to $-$.
    By construction, $\psi(\ba) = M$.
    Therefore, $\psi$ is a bijection.
\end{proof}

To help illustrate the surjectivity of $\psi$, consider the matchings illustrated in Figure~\ref{fig:psi example}.
The matching $M$ on the top contains the edge $14$, so we set $a_1 = 4$ and $a_4 = 1$.
Since there is an edge covering $w_{2,5}$, we set $a_2 = 5$.
But $3$ is not covered at all, so we set $a_3 = 4$.
Similarly, we set $a_5 = 3$.
We can immediately set $a_{i,j} = -$ for all $ij \neq 14$ since, for each of these, if $a_i = j$, then $a_j \neq i$.
It turns out that $a_{1,4} = -$ as well, but this time because $M$ contains the edge $14$.
Thus, $\ba := \psi^{-1}(M) = (4,5,4,1,3,-,-,-,-,-,-)$.

For the second matching $M'$ in Figure~\ref{fig:psi example}, we obtain the same values for $a_1,\dots,a_5$ and the same signs for all $a_{i,j}$ with $ij \neq 14$.
This time, since $14 \notin M'$, we use $a_{1,4} = +$, giving that
$\ba' :=\psi^{-1}(M') = (4,5,4,1,3,+,-,-,-,-,-)$.
Note that the clique vectors $\ba$ and $\ba'$ are exactly those we started with in Example~\ref{ex:k32cliquevectors}.

\begin{corollary}\label{cor:cliquematchingbijection}
    There is a bijection between $\cliques(H)$ and $\calM(H)$.    
\end{corollary}

\begin{proof}
    The map $\psi\circ\phi^{-1}$ yields the desired bijection.
\end{proof}

\section{Ehrhart $h^*$-Polynomials and Matching Polynomials}\label{sec:ehrhartbijection}

In this section, we use the bijections established in the previous section to prove Theorem~\ref{thm:matchingequalshstar}.
We begin by reviewing the theory of framing lattices, which were introduced by von~Bell and Ceballos~\cite{framinglattices}.
Suppose two routes $R_1$ and $R_2$ of $G(H)$ are incoherent at some vertex $v$.
If this incoherence occurs because $R_1v < R_2v$ and $vR_2 < vR_1$, then we say $R_1$ is \emph{clockwise}, or simply \emph{cw}, from $R_2$ at $v$.
We write this more compactly as $R_1 <_v^{\cw} R_2$.
Additionally, we say that $R_2$ can be obtained from $R_1$ via a \emph{counterclockwise}, or \emph{ccw}, \emph{rotation} at $v$.
Figure~\ref{fig:incoherence} illustrates this using two routes in $G(K_{2,2})$.

One may use these notions to impose a partial order on $\cliques=\cliques(H)$, the maximal cliques of $G(H)$. 
If $C_1, C_2 \in \cliques$, we say $C_2$ is obtained by a \emph{ccw rotation of} $C_1$ if it is obtained by applying a ccw rotation to a single route in $C_1$ at some vertex.
More generally, we define $C_1 \leq_{\rot}^{\ccw} C_2$ if $C_2$ can be obtained from $C_1$ through a sequence of ccw rotations.
This poset, denoted $\scrL_{G,F} = (\cliques, \leq_{\rot}^{\ccw})$, is called the \emph{framing lattice} of the framed graph $(G,F)$.
While von~Bell and Ceballos proved that these are lattices~\cite{framinglattices}, we will only require the fact that this relation yields a partial order on the set of cliques.
Typically we write only $C_1 \leq C_2$ when the meaning is clear from context. 
Since $\scrL_{G,F}$ is a poset on $\cliques$, the set of matchings $(\psi \circ \phi^{-1})(\scrL_{G,F})$ admits an isomorphic poset structure.

\begin{figure}
\begin{center}
\begin{tikzpicture}[scale=2]
\vertex[fill](a11) at (0,1.6) {};
\vertex[fill,label=above:\scriptsize{$1$}](a12) at (0,1.7) {};
\vertex[fill,label=below:\scriptsize{$2$}](a2) at (0,0.4) {};
\vertex[fill,label=above:\scriptsize{$3$}](a42) at (1.5,1.7) {};
\vertex[fill](a51) at (1.5,0.5) {};
\vertex[fill,label=below:\scriptsize{$4$}](a52) at (1.5,0.4) {};
\vertex[fill,label=left:\scriptsize{$s$}](s) at (-1,1) {};
\vertex[fill,label=right:\scriptsize{$t$}](t) at (2.5,1) {};

\draw[-stealth, thick,dashed] (s) to[bend left=30] (a12);
\draw[-stealth, thick] (s) to[bend right=30] (a11);

\draw[-stealth, thick,dashed] (a11) to (a52);
\draw[-stealth, thick] (a12) to (a51);

\draw[-stealth, thick,dashed] (a52) to[bend right=30] (t);
\draw[-stealth, thick] (a51) to[bend left=30] (t);

\end{tikzpicture}
\end{center}
\caption{Two routes, $R_2$ (solid) and $R_1$ (dashed), in $G(K_{2,2})$.
The routes are incoherent at both~$1$ and~$4$  and $R_1$ is cw from $R_2$.}
\label{fig:incoherence}
\end{figure}
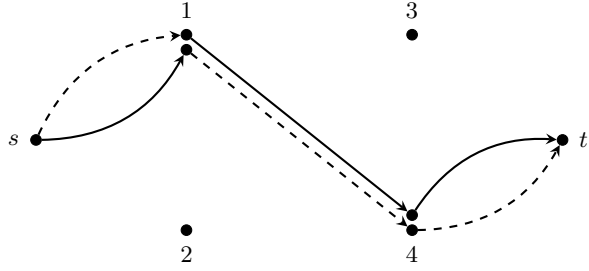

Our proof of Theorem~\ref{thm:matchingequalshstar} will depend on the following lemma that modifies a triangulation to represent a polytope as a disjoint union of half-open simplices.

\begin{lemma}\label{lem:disjointunion}
    Let $P$ be a polytope with a triangulation $\calT$, and suppose that the dual graph of $\calT$ is the Hasse diagram of a poset $Q_\calT$. 
    Let $\Sigma(\calT)$ be the simplicial complex with maximal simplices $\calT$.
    Suppose that for each $S\in\Sigma(\calT)$, there is a unique facet $C_{\max}(S)\in\calT$ with the property that for any $\Delta\in\calT$ with $S\subseteq\Delta$, there exists a saturated chain in $Q_\calT$ of the form 
    \[
        \Delta = \Delta_0 \lessdot \Delta_1 \lessdot \cdots \lessdot \Delta_k = C_{\max}(S) 
    \]
    with $S\subseteq\Delta_i$ for every $i$.
    Then the collection
    \[
        \calT^{\star} = \left\{\Delta \setminus \bigcup_{\Delta \lessdot_{Q_\calT} \Delta'} \Delta' : \Delta \in \calT\right\} 
    \]
    creates a disjoint partition of $P$.
    We refer to the set 
    \[
    \Delta \setminus \bigcup_{\Delta \lessdot \Delta'} \Delta'
    \]
    as a \emph{half-open simplex} in $\calT^\star$ and we refer to the set $\calT^\star$ as a \emph{half-open triangulation}.
\end{lemma}

\begin{proof}
    We must show that for any face $S\in\Sigma(\calT)$, the relative interior of $S$ is contained in a unique half-open simplex.
    Let $S\subseteq\Delta\neq C_{\max}(S)$.
    Then there exists a saturated chain 
    \[
        \Delta = \Delta_0 \lessdot \Delta_1 \lessdot \cdots \lessdot \Delta_k = C_{\max}(S),
    \]
    where $S\subseteq\Delta_1$.
    In $\calT^\star$, the points in $\Delta_1$ are removed when producing the half-open simplex corresponding to $\Delta$.
    Hence, $S$ is not contained in $\Delta$ for any $\Delta\neq C_{\max}(S)$.
    Since every point in $P$ is contained in the relative interior of a unique $S\in\Sigma(\calT)$, every point in $P$ is contained in a unique half-open simplex in $\calT^\star$, as desired. 
\end{proof}

We will be applying Lemma~\ref{lem:disjointunion} to DKK triangulations, which are unimodular.
The following lemma, which follows immediately from~\cite[Section 3.1]{kvbarvinok}, provides the link we will need to Ehrhart theory.

\begin{lemma}\label{lem:ehrharthalfopensimplex}
    Let $\Delta$ be a full-dimensional unimodular simplex in $\R^d$.
    Denote by $\Delta^{(k)}$ the half-open unimodular simplex obtained from $\Delta$ by removing $k$ facets.
    The Ehrhart series of $\Delta^{(k)}$ is $\frac{z^k}{(1-z)^{d+1}}$.
\end{lemma}

Given a finite poset $Q$, for an element $q\in Q$, let $\cov_Q(q)$ denote the number of elements in $Q$ that cover $q$.
Recall that, for a triangulation $\calT$, the graph of $\calT$ is the graph whose vertices are the simplices in $\calT$ and $\Delta,\Delta' \in \calT$ form an edge if and only if $\Delta$ and $\Delta'$ share a facet. 

\begin{corollary}\label{cor:hstarcovering}
    Suppose $P$ is a lattice polytope with a unimodular triangulation $\calT$ such that the dual graph of $\calT$ is the Hasse diagram of a poset $Q_\calT$ satisfying the hypotheses of Lemma~\ref{lem:disjointunion}.
    Then
    \[
    h^*(P;z)=\sum_{\Delta\in\calT}z^{\cov_{Q_\calT}(\Delta)} \, .
    \]
\end{corollary}

\begin{proof}
    The $h^*$-polynomial of $P$ is the sum of the $h^*$-polynomials of the half-open simplices in $\calT^\star$, since $\calT^\star$ is a disjoint union of half-open simplices having the same dimension as $P$.
    The number of facets removed from $\Delta\in\calT$ to produce the corresponding half-open simplex in $\calT^\star$ is $\cov_{Q_\calT}(\Delta)$, and this completes the proof.
\end{proof}

To finish the proof of Theorem~\ref{thm:matchingequalshstar}, we need to show two things regarding $G=G(H)$.
First, we need to show that the DKK triangulation of $\calF_1(G)$ satisfies the hypotheses of Lemma~\ref{lem:disjointunion} when using the framing lattice $\scrL_{G,F}$ for the canonical bipartite framing.
Second, we need to show that the number of cliques covering $C$ in $\scrL_{G,F}$ is given by the number of edges in the matching $\psi\circ\phi^{-1}(C)$, i.e., that
\[
\cov_{\scrL_{G,F}}(C)=|\psi\circ\phi^{-1}(C)| \, .
\]
The following proposition is our key tool for this endeavor, as it characterizes the cover relation in the framing lattice for $G(H)$ in the context of clique vectors.

\begin{proposition}\label{prop:cover}
    Let $\ba, \bb$ be clique vectors of $G(H)$ with respect to the canonical bipartite framing.
    Then $\phi(\ba) \lessdot \phi(\bb)$ if and only if exactly one of the following hold:
    \begin{enumerate}
        \item $\ba$ and $\bb$ differ in exactly one entry $a_i \neq b_i$, and
            \begin{enumerate}
                \item $i \in \{1,\dots,n\}$,
                \item $a_i = j$ and $a_j \neq i$, and
                \item $b_i$ is the largest neighbor of $i$ less than $j$.
            \end{enumerate}
        \item $\ba$ and $\bb$ differ in exactly one entry $a_j \neq b_j$, and
            \begin{enumerate}
                \item $j \in \{n+1,\dots,n+m\}$,
                \item $a_i \neq j$ and $a_j = i$, and
                \item $b_j$ is the smallest neighbor of $j$ greater than $i$.
            \end{enumerate}
        \item $\ba$ and $\bb$ differ only in that $a_{i,j} = -$ while $b_{i,j} = +$.
        \item $\ba$ and $\bb$ differ in exactly two entries as follows:
            \begin{enumerate}
                \item $a_{i,j} = +$ while $b_{i,j} = -$, and
                \item either $b_i$ is the largest neighbor of $i$ that is less than $j$ or $b_j$ is the smallest neighbor of $j$ that is greater than $i$.
            \end{enumerate}
    \end{enumerate}
\end{proposition}

\begin{proof}
    First suppose $\phi(\ba) \lessdot \phi(\bb)$ and let $S = \phi(\ba) \cap \phi(\bb)$.
    This means $\phi(\ba)$ and $\phi(\bb)$ are maximal cliques of $G(H)$ such that
    \begin{equation}\label{eq: clique difference}
       S = \phi(\ba) \setminus \{R_\ba\} = \phi(\bb) \setminus \{R_\bb\}
    \end{equation}
    for some routes $R_\ba \in \phi(\ba)$ and $R_\bb \in \phi(\bb)$.
    These routes satisfy $R_\ba <^{\cw} R_\bb$.
    This can occur in three ways:
    \begin{enumerate}
        \item[(A)] $R_\ba = \alpha_{1,i}\beta_{i,j}\gamma_{j,2}$ and $R_\bb = \alpha_{2,i}\beta_{i,j}\gamma_{j,1}$ for some edge $ij \in H$;
        \item[(B)] $R_\ba = \alpha_{1,i}\beta_{i,j'}\gamma_{j',*}$ and $R_\bb = \alpha_{2,i}\beta_{i,j}\gamma_{j,*}$ for some $n+1 \leq j < j' \leq n+m$; or
        \item[(C)] $R_\ba = \alpha_{*,i}\beta_{i,j}\gamma_{j,2}$ and $R_\bb = \alpha_{*,i'}\beta_{i',j}\gamma_{j,1}$ for some $1 \leq i < i' \leq n$. 
    \end{enumerate}

    \textbf{Case (A)}: if $\beta_{i,j}$ is used once in $\phi(\ba)$ and $\phi(\bb)$, then $a_i > j$ and $b_i < j$.
    However, by \eqref{eq: clique difference}, this means the set of routes passing through $\beta_{i,a_i}$ is the same in $\phi(\ba)$ and $\phi(\bb)$, hence $a_i = b_i$, which is impossible. 
    
    Now suppose $\beta_{i,j}$ is used twice in each of $\phi(\ba)$ and $\phi(\bb)$.
    By \eqref{eq: clique difference}, either both $a_i = j$ and $a_j< i$ or both $a_i>j$ and and $a_j = i$. 
    Without loss of generality, suppose the latter. 
    This means $b_i = j$ and $b_j > i$.
    However, this means $\beta_{i,a_i}$ is used a different number of times in $\phi(\ba)$ and in $\phi(\bb)$, contradicting \eqref{eq: clique difference}.

    Lastly, suppose $\beta_{i,j}$ appears three times in each of $\phi(\ba)$ and $\phi(\bb)$.
    Then $a_i = b_i = j$, $a_j = b_j = i$, and $a_{i,j} = -b_{i,j} = -$.
    By definition, both $\ba$ and $\bb$ are valid clique vectors.
    This falls into condition \emph{(3)} of the proposition statement.

    \textbf{Case (B)}: Since $\phi(\bb)$ is obtained from $\phi(\ba)$, $\beta_{i,j}$ is part of at most two routes in $\phi(\ba)$ and $\beta_{i,j'}$ is part of at least two routes in $\phi(\ba)$.
    First suppose $\beta_{i,j}$ appears in one route in $\phi(\ba)$ and $\beta_{i,j'}$ is in two.
    One way this can happen is if $a_i > j'$ and $a_{j'} = i$.
    When this happens, $b_{j'} \neq i$.
    However, \eqref{eq: clique difference} implies $b_{j'} \neq k$ for all $k \neq i$ as well, so this cannot be the case.

    The other way in which $\beta_{i,j}$ appears once and $\beta_{i,j'}$ appears twice is when $a_i = j'$ and $a_{j'} \neq i$.
    In each of the four cases of $a_j,a_{j'}$ being greater than or less than $i$, we find that $b_i = j$, $b_k = a_k$ for all $k \neq i$, and $b_{k,l} = a_{k,l}$ for all $k,l$.
    As a result, $\bb$ is a valid clique vector.

    In order for $\phi(\bb)$ to cover $\phi(\ba)$, we must have that $j$ is the largest neighbor of $i$ that is less than $j'$. 
    This falls into condition \emph{(1)} of the statement of the proposition.

    Next suppose that $\beta_{i,j}$ and $\beta_{i,j'}$ each appear in exactly two routes in $\phi(\ba)$. 
    This occurs when $a_i = b_i = j$, $b_j = i$, $a_j \neq i$, and $a_{j'} = i$.
    As a result, one route in $\phi(\ba)$ containing $\beta_{i,j}$ must begin with $\alpha_{1,j}$ and the other begins with $\alpha_{2,i}$, and both routes containing $\beta_{i,j'}$ must both begin with $\alpha_{2,i}$.
    However, this does not allow any route to be chosen as $R_\ba$ in a way that is consistent with the conditions assumed in case (B).
    So, this cannot occur.

    Now suppose $\beta_{i,j}$ is used in one route of $\phi(\ba)$ and $\beta_{i,j'}$ is used in three.
    This means $a_i = j'$, $a_{j'} = i$, and $a_j \neq i$.
    In order for $R_\bb$ to be coherent with the remaining routes passing through $i$, it must hold that $a_{i,j'} = +$, $b_i < j'$, and $b_{j'} = i$.
    Consequently, $b_i = j$ and $b_j = a_j$, forcing $b_{i,j'} = -$ and all other entries of $\bb$ are the same as those of $\ba$. 
    For $\phi(\bb)$ to cover $\phi(\ba)$, $j$ must be the largest neighbor of $i$ that is smaller than $j'$.
    It follows that $\bb$ is a valid clique vector, falling into condition \emph{(4)} of the proposition statement.

    Lastly, if $\beta_{i,j}$ is used twice and $\beta_{i,j'}$ is used three times in $\phi(\ba)$, then $a_i = j'$, $a_j = a_{j'} = i$, and $a_{i,j'} = +$.
    This forces $b_i = j$, $b_j = b_{j'} = i$, and $b_{i,j} = -$.
    This again results in a $\bb$ being a valid clique vector and falling into condition \emph{(4)} of the proposition statement. 

    \textbf{Case (C)}: This case is symmetric to Case (B).
    In particular, this case accounts for condition \emph{(2)} in the statement of the proposition and for the remaining part of condition \emph{(4)}.

    To establish the converse, suppose $\ba$ and $\bb$ satisfy one of the conditions in the statement of the proposition.
    The argument for verifying $\phi(\ba) \lessdot \phi(\bb)$ is similar enough for each condition, so we only present the argument for one of them.
    
    Suppose condition \emph{(1)} is satisfied, and without loss of generality suppose $a_j < i$.
    By Definition~\ref{def:cliquemap}, the clique $\phi(\ba)$ will contain the routes
    \[
        \alpha_{1,i}\beta_{i,j}\gamma_{j,2}~\text{ and }~\alpha_{2,i}\beta_{i,j}\gamma_{j,2},
    \]
    and $\phi(\bb)$ will only contain $\alpha_{2,i}\beta_{i,j}\gamma_{j,2}$.
    It is possible to have $a_k = a_{b_i} < i$, $a_k = i$, or $a_k > i$; again without loss of generality, suppose $a_k < i$.
    This means the only route containing $\beta_{i,k}$ in $\phi(\ba)$ is $\alpha_{1,i}\beta_{i,k}\gamma_{k,2}$, while $\phi(\bb)$ contains the routes
    \[
        \alpha_{1,i}\beta_{i,k}\gamma_{k,2}~\text{ and }~\alpha_{2,i}\beta_{i,k}\gamma_{k,2}.
    \]
    
    Since $a_i$ and $b_i$ are the only entries of $\ba$ and $\bb$ which disagree, all other routes of $\phi(\ba)$ and $\phi(\bb)$ are the same.
    That is,
    \[
        \phi(\ba) \setminus \phi(\bb) = \{\alpha_{1,i}\beta_{i,j}\gamma_{j,2}\}~\text{ and }~\phi(\bb) \setminus \phi(\ba) = \{\alpha_{2,i}\beta_{i,k}\gamma_{k,2}\}.
    \]
    Since $\alpha_{1,i}\beta_{i,j}\gamma_{j,2} <_i^{\cw} \alpha_{2,i}\beta_{i,k}\gamma_{k,2}$, it follows that $\phi(\ba) \lessdot \phi(\bb)$.
\end{proof}

Proposition~\ref{prop:cover} has a combinatorial interpretation in terms of the matching $\psi(\ba)$.
For leaf edges of $W(H)$ incident to $i\in\{1,2,\ldots,n\}$, condition \emph{(1)} in the proposition states that moving the leaf edge ``up'' to the next leaf edge will result in a cover relation, unless doing so would result in the presence of the leaves $iw_{i,j}$ and $jw_{j,i}$, in which case those leaves are removed and the edge $ij$ is included.
Note that this might include situations where the updated matching has an ``absent'' leaf, such as the non-existence of the edge $1w_{1,4}$ in the matching $\psi(\ba')$ from Example~\ref{ex:matchingmap}.

For leaf edges of $W(H)$ incident to one of $j\in\{n+1,n+2,\ldots,n+m\}$, condition \emph{(2)} in the proposition states that moving the leaf edge ``down'' to the next leaf edge will result in a cover relation, unless doing so would result in the presence of the leaves $iw_{i,j}$ and $jw_{j,i}$, in which case those leaves are removed and the edge $ij$ is included.
Note that as with condition \emph{(1)}, this might include situations where the updated matching has an ``absent'' leaf.

Condition \emph{(3)} states that if the edge $ij\in E(H)$ is present in the matching, then it can be removed and replaced by the leaf edges $iw_{i,j}$ and $jw_{j,i}$; however, it is possible that one of these edges might correspond to an ``absent'' leaf such as mentioned with condition \emph{(1)}, in which case that leaf is simply omitted.
Finally, condition \emph{(4)} states that in the situation where the matching contains a pair of leaf edges $iw_{i,j}$ and $jw_{j,i}$, where we must also consider the possibility of one of these leaves being an ``absent'' leaf in $W(H)$, then either the left leaf edge can be moved ``up'' or the right leaf edge can be moved ``down''.

\begin{example}
    Consider the matching $\psi(\ba)$ from Example~\ref{ex:matchingmap}.
    For this matching, the edge $14$ could be replaced by the edge $4w_{4,1}$, where we omit the ``absent'' leaf $1w_{1,4}$ that would otherwise be indicated.
    This would correspond to changing $a_{1,4}$ from $-$ to $+$.
    
    Another option would be that $2w_{2,5}$ could be removed; it would be expected to be replaced by $2w_{2,4}$, but this is an ``absent'' leaf in $W(H)$ and therefore is omitted.
    This would be the result of moving $2w_{2,5}$ ``up'' among the leaf vertices, and would correspond to changing $a_2$ from $5$ to $4$.

    Observe that there is no value to replace $a_3=4$ with as there is no value in $\{4,5\}$ that is less than $4$. 
    This is mirrored in the matching by the fact that the leaf edge at vertex $3$ is already missing.
    
    Similarly, there is no option to shift the leaf edges at vertex $5$, since $a_5=3$ and there is no value in $\{1,2,3\}$ greater than $3$, which corresponds in the matching to the fact that there is no leaf edge present in the matching.

    Note that in the matching $\psi(\ba')$ from Example~\ref{ex:matchingmap}, it is now possible to move $4w_{4,1}$ to $4w_{4,2}$, which would correspond to changing $a_{1,4}$ from $+$ to $-$ and changing $a_4$ from $1$ to $2$.

    Each of the above adjustments to matchings corresponds to a cover relation among their associated clique vectors.
\end{example}

\begin{corollary}\label{cor:dualgraphsatisfieslemma}
Let $G = G(H)$ and let $F$ be the canonical bipartite framing for $G$.
The poset $\scrL_{G,F}$ satisfies the hypotheses of Lemma~\ref{lem:disjointunion}.
\end{corollary}

\begin{proof}
    Suppose that $S$ is a set of coherent routes in $G(H)$ and that $S\subseteq \phi(\ba)$ for a clique vector $\ba$.
    We begin by observing that for each $i\in \{1,2,\ldots,n+m\}$, the routes in $S$ lead to values $\min_S(i)$ and $\max_S(i)$ such that 
    \[
    \mathrm{min}_S(i)\leq a_i\leq \mathrm{max}_S(i) \, .
    \]
    There are four ways that these bounds on $a_i$ are induced by routes in $S$.
    First, suppose $S$ contains the route $\alpha_{2,i}\beta_{i,j}\gamma_{j,2}$.
    Then we must have $a_i\leq j$ and $a_j\leq i$, as otherwise $\phi(\ba)$ would contain a conflicting route.
    Second, and by an identical argument, if $S$ contains $\alpha_{2,i}\beta_{i,j}\gamma_{j,1}$, then $a_i\leq j$ and $a_j\geq i$, and further if both $a_i=j$ and $a_j=i$, then $a_{i,j}=+$.
    Third, if $S$ contains $\alpha_{1,i}\beta_{i,j}\gamma_{j,2}$, then $a_i\geq j$ and $a_j\leq i$, and further if both $a_i=j$ and $a_j=i$, then $a_{i,j}=-$.
    Fourth and finally, note that if $S$ contains $\alpha_{1,i}\beta_{i,j}\gamma_{j,1}$, then $a_i\geq j$ and $a_j\geq i$.
    Thus, there exists upper and lower bounds on each $a_i$ for $i\in\{1,2,\ldots,n+m\}$.
    Further, observe that in the second and third cases above, conditions on the signs of $a_{i,j}$ are forced by the routes in $S$ when $a_i=j$ and $a_j=i$.

    We next define a clique vector $\ba^S$ and prove that it satisfies $\phi(\ba^S)=C_{max}(S)$ with all the corresponding properties of the hypotheses of Lemma~\ref{lem:disjointunion}.
    For each $i\in \{1,2,\ldots,n\}$, set $a^S_i:=\min_S(i)$.
    For each $j\in \{n+1,n+2,\ldots,n+m\}$, set $a^S_j:=\max_S(j)$.
    If we further have
    \[
    i=\mathrm{max}_S(j) \text{ and } j=\mathrm{min}_S(i) \, ,
    \]
    then $a^S_{i,j}:=+$.

    Now suppose that $\ba'$ is a clique vector such that $S\subseteq \phi(\ba')$.
    Then we know for every $i\in \{1,2,\ldots,n+m\}$ we have
    \[
    \mathrm{min}_S(i)\leq a'_i\leq \mathrm{max}_S(i) \, .
    \]
    Further, we know that for every $j$ such that $\mathrm{min}_S(i)< j < \mathrm{max}_S(i)$, there is no route in $S$ supported on the edge $\beta_{i,j}$, as this would lead to a contradiction to the max and min values of $a_i$ via the argument given in the first paragraph above.
    Therefore, we are freely able to adjust the entries of $\ba'$ as described in the four conditions of Proposition~\ref{prop:cover}. 
    By iteratively decreasing the values $a'_i$ and increasing the values $a'_j$, applying conditions \emph{(3)} and \emph{(4)} as needed during this process, a sequence of clique vectors is produced that terminates at $\ba^S$.
    There are no routes in $S$ that cause conflicts with any of the maximal cliques corresponding to these clique vectors, as there are no routes to conflict with along those edges.
    This produces the desired sequence of cover relations in the framing lattice.
    Since each of these saturated chains in the framing lattice terminates at $\phi(\ba^S)$, this completes the proof.
    \end{proof}

\begin{corollary}\label{cor:coverings of a clique}
    Let $G = G(H)$ and let $F$ be the canonical bipartite framing for $G$.
    Given a clique $C\in\scrL_{G,F}$ , we have 
    \[
    \cov_{\scrL_{G,F}}(C)=|\psi\circ\phi^{-1}(C)| \, ,
    \]
    i.e., the number of cliques covering $C$ is equal to the number of edges in the matching in $W(H)$ corresponding to $C$.
\end{corollary}

\begin{proof}
    Let $\phi(\ba)$ be a clique with corresponding matching $\psi(\ba)$ of $W(H)$.
    We will prove this corollary by showing how each cover relation listed in Proposition~\ref{prop:cover} can be associated to performing a unique ``action'' on a specific edge of $\psi(\ba)$. 
    
    Suppose $iw_{i,j} \in \psi(\ba)$ for some $i = 1,\dots,n$.
    There are two possibilities: $a_i = j$ and $a_j \neq i$, or $a_i = j$, $a_j = i$, and $a_{i,j} = +$.
    First, if $a_j \neq i$, then let $j'$ be the largest neighbor of $i$ that is less than $j$.
    If $j'w_{i,j'}$ is an edge of $\psi(\ba)$ as well, then let $\psi(\bb)$ be the matching that deletes $iw_{i,j}$ and $j'w_{i,j'}$ and adds $ij'$.
    Then $\phi(\ba) \lessdot \phi(\bb)$ since $\ba$ and $\bb$ satisfy Proposition~\ref{prop:cover}~\emph{(1)}. 
    If $j'w_{i,j'}$ is not an edge in $\psi(\ba)$, then replace $iw_{i,j}$ with $iw_{i,j'}$ if possible; if not, delete $iw_{i,j}$.
    If $iw_{i,j'}$ exists, then the resulting matching, $\psi(\bb)$, has a corresponding clique $\phi(\bb)$ covering $\phi(\ba)$ since $\ba$ and $\bb$ again satisfy Proposition~\ref{prop:cover}~\emph{(1)}.
    The only way $iw_{i,j'}$ cannot exist is if $j'$ and $j$ are the lowest-valued neighbors of $i$.
    Then 
    \[
        \phi(\ba) \lessdot \phi(\bb) = (\phi \circ \psi^{-1})(\psi(\ba) \setminus iw_{i,j})
    \]
    because $\ba$ and $\bb$ once more satisfy Proposition~\ref{prop:cover}~\emph{(1)}.

    Now, suppose $a_j = i$ and $a_{i,j} = +$.
    There is then exactly one way to produce a clique covering $\phi(\ba)$ whose corresponding matching modifies $iw_{i,j}$: replace $a_{i,j} = +$ with $-$ and, if possible, replace $iw_{i,j}$ with $iw_{i,j'}$ where $j'$ is the largest neighbor of $i$ that is smaller than $j$; otherwise, delete $iw_{i,j}$.
    Then the clique vectors associated to $\psi(\ba)$ and the resulting matching satisfy Proposition~\ref{prop:cover}~\emph{(4)}.

    We have shown thus far that if $\psi(\ba)$ contains an edge $iw_{i,j}$, then there is exactly one way to associate a clique $\phi(\bb)$ covering $\phi(\ba)$ by performing an action on $iw_{i,j}$.
    An analogous argument holds when considering edges of the form $jw_{j,i}$ for $j \in \{n+1,\dots,n+m\}$.
    In this case, we use Proposition~\ref{prop:cover}~\emph{(2)} or~\emph{(4)}.
    Note that if an action on $iw_{i,j}$ results in deleting an edge $j'w_{j',i}$, then performing an action on $j'w_{j',i}$ cannot result in deleting $iw_{i,j}$, so there is indeed just one action associated to each edge. 
    
    Finally, if $ij \in \psi(\ba)$, then $a_i = j$, $a_j = i$, and $a_{i,j} = -$.
    By Proposition~\ref{prop:cover}~\emph{(3)}, the clique $\phi(\ba)$ is covered by the matching corresponding to the clique vector that replaces $a_{i,j} = -$ with $+$.

    By examining each of the conditions listed in Proposition~\ref{prop:cover}, one may verify that each of the cover relations describes exactly one of the above actions on an edge of $\psi(\ba)$.
    The conclusion of the corollary follows.
\end{proof}

The following corollary is directly implied by Corollary~\ref{cor:coverings of a clique}.

\begin{corollary}\label{cor:coveringsequalsmatchings}
    Let $G = G(H)$ and let $F$ be the canonical bipartite framing for $G$.
    The number of cliques in $\scrL_{G,F}$ that are covered by $k$ other cliques is the number of $k$-matchings in $W(H)$.
\end{corollary}

We conclude this work with the proof of Theorem~\ref{thm:matchingequalshstar}.

\begin{proof}[Proof of Theorem~\ref{thm:matchingequalshstar}]
    Let $\calT$ be the DKK triangulation of $\calF_1(G)$ induced from $\scrL_{G,F}$, and let $\calT^{\star}$ be the half-open triangulation described in Lemma~\ref{lem:disjointunion}.
    By Corollary~\ref{cor:hstarcovering},
    \[
        h^*(\calF_1(G);z) = \sum_{\Delta \in \calT^{\star}} z^{\cov(\Delta)} = \sum_{k=0}^{n+m} \sum_{\substack{\Delta \in \calT^{\star} \\ \cov(\Delta) = k}} z^k.
    \]
    By Corollary~\ref{cor:coveringsequalsmatchings}, the coefficient of $z^k$ will be the number of matchings of $W(H)$ of size $k$, completing the proof.
\end{proof}

\bibliographystyle{plain}
\bibliography{refs}

\end{document}